\documentclass[10pt]{article}
\usepackage[utf8]{inputenc}
\usepackage[english]{babel}
\usepackage{comment}
\usepackage[dvipsnames,svgnames,x11names]{xcolor}

\usepackage{makeidx}
\makeindex
\usepackage{tikz}
\usepackage{bm}
\usetikzlibrary{arrows.meta}
\usetikzlibrary{shapes.multipart, positioning}
\usepackage{multicol}
\usepackage{amsfonts,amsthm,amsmath,amssymb,enumitem}
\usepackage{caption}
\usepackage{graphicx,float,subcaption}
\usepackage{verbatim}
\usepackage{amsthm}
\usepackage{multirow}

\usepackage{graphicx}
\usepackage{caption}
\usepackage{booktabs}
\usepackage{hyperref}
\usepackage{threeparttable}

\graphicspath{{Figures/}}
\hypersetup{colorlinks=true,linkcolor=black,anchorcolor=black,citecolor=black}
\usepackage{algorithm}
\usepackage{algpseudocode}
\usepackage{tikz}
\usetikzlibrary{arrows.meta,calc,positioning,fit,backgrounds}
\usepackage{array}
\usepackage{siunitx}
\newcolumntype{C}{>{\centering\arraybackslash}p{1.2cm}}
\newcolumntype{R}{>{\raggedleft\arraybackslash}p{1.5cm}}

\newtheorem{theorem}{Theorem}[section]
\newtheorem{lemma}{Lemma}[section]
\newtheorem{assumption}{Assumption}[section]
\newtheorem{proposition}[theorem]{Proposition}

\newtheorem{remark}{Remark}[section]
\newtheorem{definition}[theorem]{Definition}
\newtheorem{corollary}[theorem]{Corollary}
\usepackage{multirow}

\numberwithin{equation}{section}
\allowdisplaybreaks
\date{}

\title{On the Uniqueness of a Nonlinear Discrete Calder\'{o}n Problem\thanks{The work of M. Deng is supported by the Hong Kong PhD Fellowship Scheme. The work of B. Jin is supported by Hong Kong RGC General Research Fund (Projects 14306423 and 14306824), the ANR / RGC Joint Research Scheme (A-CUHK402/24), and a start-up fund from The Chinese University of Hong Kong.}}
\author{Elena Beretta\thanks{Division of Science, New York University Abu Dhabi, Saadiyat Island, United Arab Emirates (\texttt{eb147@nyu.edu, ag189@nyu.edu})} \and Maolin Deng\thanks{Department of Mathematics, The Chinese University of Hong Kong, Shatin, N.T., Hong Kong (\texttt{mldeng@link.cuhk.edu.hk, b.jin@cuhk.edu.hk})} \and Alberto Gandolfi\footnotemark[2] \and Bangti Jin\footnotemark[3]}

\begin{document}
\maketitle

\begin{abstract}
The discrete Calder\'on problem aims at recovering the conductivity on the
edges of a graph from boundary measurements, encoded by the discrete Dirichlet-to-Neumann (DtN) map. The problem has been intensively studied in the linear case on square lattices since the seminal works of Curtis and Morrow. In this work, we investigate a nonlinear analogue of the discrete Calder\'on problem
for a semilinear second-order elliptic equation on square lattices. We establish three uniqueness results for the conductivity recovery. First, we show that conductivity-dependent corner excitations allow a layer-by-layer reconstruction of the conductivity. Second, we study the linearization of the nonlinear DtN map at an arbitrary background boundary datum and prove that the conductivity and the background potential are uniquely determined by one pair of nonlinear Cauchy data and the linearized DtN map. Third, we show that the linearized data can be replaced by finitely many nonlinear measurements, which
uniquely determine the conductivity.
\end{abstract}


\section{Introduction}
The Calder\'{o}n problem is concerned with the recovery of an unknown conductivity distribution in the domain from measurements taken on the boundary. Since the seminal work \cite{Calderon:1980} of Alberto Calder\'{o}n, the uniqueness issue has been extensively studied \cite{FeldmanSaloUhlmann:2025}. In the literature there is also a discrete version of the Calder\'{o}n problem posed on graphs, which concerns recovering the conductivity on the edges of the graph from measurements taken at the boundary nodes. The discrete version was first studied by Curtis and Morrow \cite{CurtisMorrow:1990,CurtisMorrow:1991}, who proved uniqueness for square lattices.

Consider a graph \(G = (E,D,\partial D)\), with $D$ and $\partial D$ being the sets of interior nodes and boundary nodes, respectively.  Let \(\overline D:=D\cup\partial D\) be the full set of nodes, and the edge set \(E\) consists of unordered pairs \(\{p, q\}\) with \(p, q \in \overline{D}\) (but not $p,q\in\partial D$ simultaneously, i.e., no edge on the boundary $\partial D$). A graph \(G=(E,D,\partial D)\) is {connected} if for every \(p,q\in \overline D\), there exists a finite path
$p=p_0,p_1,\dots,p_m=q$ in \(\overline D\) such that \(\{p_{j-1},p_j\}\in E\) for all \(j=1,\dots,m\). The edge structure naturally implies an isotropic conductivity on the graph $G$, i.e., the conductance between any two connected nodes is direction-independent. For any node \(p \in D\), we define its neighborhood $N(p)$ by $N(p) := \big\{ q \in \overline{D} : \{p, q\} \in E \big\}$. We assume that every boundary node has exactly one neighbor, so \(N(p)\) is a singleton for every \(p\in\partial D\). For a set \(S\), be it \(D\), \(\partial D\), or \(E\), the notation \(\mathbb{R}^{S}\) denotes the space of real-valued functions on \(S\), and \(\mathbb{R}_+ = (0, +\infty)\).

Let \(\gamma \in \mathbb{R}_+^{E}\) denote the conductances on the edges $E$, with \({\gamma}_{pq} \in\mathbb{R}_+\) being the conductance of the edge \(\{p,q\} \in E\). Given a potential $u\in \mathbb{R}^{\overline D}$, the discrete Laplacian \(L_{\gamma}\) at a node \(p \in D\) is defined by
\[
(L_{\gamma} u)_p = \sum_{q \in N(p)} {\gamma}_{pq}(u_q - u_p),
\]
which is the net current flowing into the node \(p\) under our sign convention. Thus, in the absence of external sources, Kirchhoff's current law is expressed by \((L_\gamma u)_p=0\).

Now let the function \(f_p : \mathbb{R} \to \mathbb{R}\), for each $p\in D$, denote a node-dependent nonlinear reaction term, and let $f=(f_p)_{p\in D}$. One common choice for $f_p$ is the power-type nonlinearity, e.g., $f_p(u_p)=cu_p^3$ with $c\geq 0$, $p\in D$.
Such cubic reaction terms are classical in excitable-media models of FitzHugh--Nagumo type, which are often used as reduced descriptions of wave propagation in neural and cardiac tissues \cite{Franzone:2014}. These applications also motivate the study of discrete versions of the model, e.g., via suitable finite-dimensional discretization. The discrete counterpart may also shed light on the corresponding continuum inverse problem: the graph setting provides a finite-dimensional framework in which the main nonlinear and inverse mechanisms can often be analyzed more explicitly, e.g., localization, layer-stripping procedures, and the interplay between the background state and the linearized map.

Consider the following semilinear Dirichlet problem for the discrete Laplacian $L_\gamma$ on a connected graph \(G = (E,D, \partial D)\): find \(u \in \mathbb{R}^{\overline{D}}\) such that
\begin{equation}
  \left\{  \begin{aligned}
  (L_{\gamma} u)_p \;-\; f_p(u_p) &= 0, \quad p \in D,\\
u_p &= \varphi_p, \quad p \in \partial D,
\end{aligned}\right.
\label{eqn:Model_Nonlinear}
\end{equation}
where \(\varphi \in \mathbb{R}^{\partial D}\) denotes the Dirichlet boundary excitation. Under mild conditions on \(f\), problem \eqref{eqn:Model_Nonlinear} admits a unique solution \(u \in \mathbb{R}^{\overline{D}}\); see Theorem \ref{thm:well-posed} for the precise statement.
Given the unique solution \(u \in \mathbb{R}^{\overline{D}}\) to problem \eqref{eqn:Model_Nonlinear}, we can define the Neumann data \(\psi \in \mathbb{R}^{\partial D}\) by
\begin{equation}\label{eqn:flux}
\psi_p = \sum_{q\in N(p)}\gamma_{pq}(u_p-u_q),\quad p\in\partial D.
\end{equation}
Physically, \(\psi\) represents the outgoing current on the boundary $\partial D$.
Note that in the presence of a nonlinear reaction term \( f \), the Neumann boundary datum \(\psi \in \mathbb{R}^{\partial D}\) depends nonlinearly on both the conductivity \(\gamma\) and the Dirichlet data \(\varphi\). We define a nonlinear Dirichlet-to-Neumann (DtN) map
\begin{equation*}
F_{\gamma}: \mathbb{R}^{\partial D} \to \mathbb{R}^{\partial D},\quad \varphi \mapsto \psi.
\end{equation*}
In the absence of the semilinear term $f$, the DtN map $F_\gamma(\varphi)$ is linear in $\varphi$, and will be denoted by $\Lambda_\gamma$.

The discrete Calder\'{o}n problem on the graph $G=(E,D,\partial D)$ reads: Given some knowledge of the DtN map $F_\gamma$, is it possible to uniquely identify the conductivity $\gamma\in \mathbb{R}_+^E$? In the linear case (i.e., $f_p\equiv0$, for any $p\in D$), Curtis and Morrow \cite{CurtisMorrow:1990,CurtisMorrow:1991} established unique determination of the conductivity $\gamma$ from the DtN matrix $\Lambda_\gamma$ for 2D square lattices. These results have been extended in several directions, e.g., circular planar graphs \cite{CurtisMorrow:1994,CurtisIngermanMorrow:1998,Ingerman:2000}, multi-dimensional non-planar graphs \cite{LamPylyavskyy:2012,DengJin:2025}, single measurement (under the monotonicity assumption) \cite{ChungBerenstein:2005}, complex-valued conductivity \cite{BoyerGarzella:2016}, optimal tomography on graphs \cite{ChungGilbert:2017}, and internal data \cite{Corbett:2025}. One powerful idea in the uniqueness analysis \cite{CurtisMorrow:1990,CurtisMorrow:1991,CurtisMorrow:1994,DengJin:2025} is the layer-by-layer recovery of the conductivity $\gamma$ using suitable specialized boundary excitations. However, the nonlinear case in this graph setting has received much less attention. The nonlinearity of the model presents substantial challenges: the DtN map $F_\gamma$ is not only nonlinear but also has intricate properties. The specialized corner excitations no longer form a linear subspace but instead form a nonlinear graph.
Hence, the existing approach of deriving corner excitations by computing the kernel of a submatrix of the DtN matrix \cite{CurtisMorrow:1990,CurtisMorrow:1991} is intractable in the nonlinear case.

In this study, we investigate the nonlinear discrete Calder\'on problem associated with problem \eqref{eqn:Model_Nonlinear} on two-dimensional square lattices: determine the conductivity \(\gamma\in\mathbb R_+^E\) from finitely many Cauchy data pairs $(\varphi^{(i)},F_\gamma(\varphi^{(i)}))_{i=1}^N$
generated by the nonlinear DtN map \(F_\gamma\), thereby extending the result of Curtis and Morrow \cite{CurtisMorrow:1990} to the nonlinear
case. We present three uniqueness results for the inverse problem, distinguished by the type of available boundary data: (i) uniqueness from
conductivity-dependent corner voltage measurements, (ii) uniqueness from one nonlinear Cauchy data pair together with the linearized
DtN map $\Lambda_\gamma(\varphi)$ at the background boundary potential $\varphi$, and (iii) uniqueness for the fully nonlinear problem from finitely many
measurements taken near an arbitrary fixed boundary datum. More precisely, for \(m=|\partial D|\) and for any normalized basis
\(\{v_1,\ldots,v_m\}\) of \(\mathbb R^{\partial D}\), the \(m+1\) measurements:
$F_\gamma(\varphi)$ and  $F_\gamma(\varphi+t v_i)$, $i=1,\ldots,m$, determine \(\gamma\) for all sufficiently small \(t>0\), uniformly for
\((\gamma,\varphi)\) in compact admissible sets. See Theorems \ref{thm:gamma-recovery}, \ref{Thm:LinearizeApproach} and \ref{thm:finite-nonlinear-data} for the precise statements. These results address multiple types of boundary measurements, and to the best of our knowledge, represent the first theoretical guarantees for the semilinear discrete Calder\'{o}n problem. The analysis follows two distinct strategies depending on the type of available data: For corner excitations and linearized data, we extend the layer-by-layer slicing argument from the linear setting \cite{CurtisMorrow:1990,CurtisMorrow:1991,DengJin:2025}. The resulting proofs are constructive and lend themselves to reconstruction algorithms, whereas for general excitations, the analysis employs a perturbative $C^1$-stability argument for injective maps with full-rank differential.

Unlike the discrete Calder\'{o}n problem, continuous Calder\'{o}n-type problems with semilinear elliptic PDEs have been extensively studied in the last three decades; see, e.g., \cite{IsakovSylvester:1994,LassasLin:2021,HarrachLin:2023,Beretta:2026}. In a bounded domain $\Omega \subset \mathbb{R}^d$, $d\geq 3$, Isakov and Sylvester \cite{IsakovSylvester:1994} proved the unique determination of $a(x,u)$ in the semilinear elliptic equation $-\Delta u+a(x,u)=0$  if $a(x,0)=0$ and the DtN map is fixed. Lassas et al. \cite{LassasLin:2021} developed a high-order linearization method for the Calder\'{o}n-type inverse problems for semilinear equations with power-type nonlinearities, and obtained several uniqueness results from the nonlinear DtN map. Recently, Beretta et al. \cite{Beretta:2026} proved the unique determination of an anisotropic conductivity $\gamma$ in the semilinear model $-\nabla\cdot(\gamma\nabla u)+\alpha u^3=0$ from the Neumann-to-Dirichlet map. The proof relies on first-order linearization and a nontrivial pacing current. One key idea in this line of inquiry is linearization, to exploit the smooth dependence of the DtN map on the perturbations of the boundary data, and to analyze the (high-order) linearization. See \cite{Lassas:2025} for an overview of the (high-order) linearization technique. In the discrete case, we also employ linearization, but the proofs are constructive, using specialized excitations that generate localized potentials, instead of relying on the density argument as in the continuous case.

The remainder of the paper is organized as follows. In Section \ref{sec:direct}, we study the well-posedness of problem \eqref{eqn:Model_Nonlinear}, which ensures the well-definedness of the nonlinear DtN map \(F_\gamma(\varphi)\).
In Section \ref{sec:inverse}, we describe the inverse problem and state the main inverse results. The uniqueness proofs are given in Sections \ref{sec:corner-data}, \ref{sec:linearized-recovery} and \ref{sec:finite-data}.
In Table~\ref{tab:notation}, we list the symbols that are frequently used below.

\begin{table}[hbt!]
 \centering
 \caption{List of symbols and notations.}
\label{tab:notation}
 \begin{tabular}{p{2.5cm}p{9.5cm}}
 \toprule
 \textbf{Notation} & \textbf{Description} \\
 \midrule
$G = (E, D,\partial D)$ & graph with edge set $E$, interior node set $D$, boundary node set $\partial D$. \\
$\overline{D}$ & the set of all nodes: $\overline{D} = D \cup \partial D$. \\
$N(p)$ & the set of neighbors of the node $p$. \\
$\gamma = [{\gamma}_{pq}]$ & conductivity vector defined on edges, with ${\gamma}_{pq} > 0$ for $\{p,q\} \in E$. \\
$\varphi$ & Dirichlet boundary condition.\\
$\psi$ & Neumann boundary condition.\\
$L_{\gamma}$ & discrete graph Laplacian for the conductivity $\gamma$.\\
$f_p$ & nonlinear function at node $p \in D$. \\
$u$ & potential defined on $\overline{D}$. \\
$F_{\gamma}$ & nonlinear Dirichlet-to-Neumann (DtN) map. \\
$X_{\leq \ell}, X_\ell, X_{>\ell}$ & diagonal layers defined in \eqref{eqn:X-ell}, $X\in \{D, \overline{D}, \partial D\}$. \\
$V_k = V_{k,\gamma}$ & the $k$th corner-excitation set. \\
$\Lambda_{\gamma}(\varphi)$ & Fr\'{e}chet derivative of $F_{\gamma}$ at $\varphi$ (linearized DtN map). \\
\bottomrule
\end{tabular}
\end{table}

\section{Well-posedness of the direct problem}\label{sec:direct}

We establish that the semilinear problem \eqref{eqn:Model_Nonlinear} has a unique solution for every Dirichlet excitation $\varphi\in\mathbb{R}^{\partial D}$, which is needed for defining the nonlinear DtN map \(F_\gamma\). The analysis uses a variational argument: the solutions are the stationary points of an energy $\mathcal{E}_\gamma$, and under a monotonicity assumption on \(f_p\), $\mathcal{E}_\gamma$ is strictly convex and coercive on an affine space of functions satisfying the prescribed Dirichlet condition:
\begin{align}\label{eqn:adm-set}
 \mathbb{U} = \{ u \in \mathbb{R}^{\overline{D}} : u|_{\partial D} = \varphi \}.
\end{align}
The existence and uniqueness then follow from the argument for finite-dimensional convex minimization.

\begin{assumption}\label{assumption:monotone-f}
For each $p\in D$, $f_p$ is a non-decreasing continuous function.
\end{assumption}

Throughout, if \(S\subset X\), $X\in \{D,\partial D,\overline{D}\}$ and \(h\in \mathbb R^X\), then \(h|_S\) denotes the restriction of
\(h\) to \(S\). We define an energy functional $\mathcal{E}_{\gamma}$ that captures the total energy of the system, composed of a dissipation term and a term associated with the nonlinear reaction:
\begin{equation}\label{eqn:energy}
   \mathcal{E}_{\gamma}(u) =\mathcal{Q}(u) + \sum_{p \in D}  \mathcal{F}_p(u_p),\quad \mbox{with } \mathcal{Q}(u)=  \frac{1}{2} \sum_{{pq}\in E} {\gamma}_{pq} (u_p - u_q)^2, \ \mathcal{F}_p(t) = \int_0^t f_p(\tau) \, \mathrm{d}\tau.
\end{equation}
 It can be shown that problem \eqref{eqn:Model_Nonlinear} is equivalent to minimizing \(\mathcal{E}_{\gamma}\) over $\mathbb{U}$.
In fact, the stationary points of $ \mathcal{E}_{\gamma}$ are exactly the solutions of problem \eqref{eqn:Model_Nonlinear}.
\begin{lemma}\label{Lem:full-energy-stationary}
Suppose that each \(f_p\) is continuous.
Then \(u\in \mathbb R^{\overline D}\) solves problem \eqref{eqn:Model_Nonlinear} if and only if
\[
u|_{\partial D}=\varphi
\quad\text{and}\quad
(\nabla \mathcal E_\gamma(u))|_D=0.
\]
\end{lemma}
\begin{proof}
Fix \(u\in \mathbb R^{\overline D}\). For any \(p\in D\), differentiating the quadratic part $\mathcal{Q}(u)$ gives
\[
\frac{\partial}{\partial u_p} \mathcal{Q}(u) = \sum_{q\in N(p)}\gamma_{pq}(u_p-u_q)
= -(L_\gamma u)_p.
\]
Moreover, by the fundamental theorem of calculus, we have
\[
\frac{\partial}{\partial u_p} \left(\sum_{r\in D} \mathcal{F}_r(u_r)\right) = \mathcal{F}_p'(u_p)=f_p(u_p).
\]
Hence, for every \(p\in D\),
\[
\frac{\partial}{\partial u_p}\mathcal E_\gamma(u) = -(L_\gamma u)_p + f_p(u_p).
\]
That is, $(\nabla \mathcal E_\gamma(u))_p=0$ if and only if $(L_\gamma u)_p-f_p(u_p)=0.$ Then it follows that $(\nabla \mathcal E_\gamma(u))|_D=0$  if and only if the interior equations in problem \eqref{eqn:Model_Nonlinear} hold for all \(p\in D\). Together with the boundary condition \(u|_{\partial D}=\varphi\), this proves the desired claim.
\end{proof}

The following convexity and coercivity results are crucial.
\begin{lemma}\label{lem:convexity-coercivity}
Let Assumption~\ref{assumption:monotone-f} hold, and let \(G\) be a connected finite graph with a nonempty boundary \(\partial D\). Then the
functional \(\mathcal{E}_\gamma\) is strictly convex and coercive on the admissible set $\mathbb{U}$.
\end{lemma}

\begin{proof}
Let \(u_D=u|_D\). Then it suffices to study the dependence of the energy \(\mathcal{E}_\gamma\) on \(u_D\).
Expanding the quadratic part \(\mathcal{Q}\) with respect to $u_D$ gives
\[
\mathcal{Q}(u)=\tfrac12 u_D^\top M_D u_D+c_1^\top u_D+c_2,
\]
where \(c_1\in \mathbb{R}^{D}\) and \(c_2\in\mathbb{R}\) depend only on \(\gamma\) and on the prescribed
boundary value \(\varphi\). The matrix \(M_D\in \mathbb{R}^{D\times D}\) is the Dirichlet graph
Laplacian, defined by
\begin{align}\label{eqn:graph-Lap}
(M_D)_{pq} = \begin{cases}
\displaystyle \sum_{r\in N(p)} {\gamma}_{pr}, & p=q,\\[6pt]
-{\gamma}_{pq}, & q\in N(p)\cap D,\\[4pt]
0, & \text{otherwise}.
\end{cases}
\end{align}
The matrix \(M_D\) is clearly symmetric.
First we show that \(M_D\) is positive definite. For
\(z\in\mathbb R^D\), let \(\widetilde z\) be its zero extension to
\(D\cup\partial D\). 
Then direct computation gives
\[
z^\top M_Dz
=
\sum_{\{p,q\}\in E}
\gamma_{pq}
\bigl(\widetilde z_p-\widetilde z_q\bigr)^2.
\]
Hence \(z^\top M_Dz\ge0\). If \(z^\top M_Dz=0\), then
$\widetilde z_p=\widetilde z_q$
for every $\{p,q\}\in E$. Since the graph \(G\) is connected, \(\widetilde z\) is constant on
\(D\cup\partial D\). Since \(\partial D\neq\varnothing\) and
\(\widetilde z=0\) on \(\partial D\), the constant is zero.
Thus \(z=0\), and \(M_D\) is positive definite. It follows that \(\mathcal Q\), as a function of \(u_D\),
is strictly convex. Moreover, since each \(f_p\) is continuous and
non-decreasing, the function
$\mathcal{F}_p(t):=\int_0^t f_p(s)\,{\rm d}s$
is convex. Therefore, the map
\[
u_D\longmapsto \sum_{p\in D}\mathcal{F}_p(u_p)
\]
is convex. Since the sum of a strictly convex function and a convex
function is strictly convex, \(\mathcal E_\gamma\) is strictly
convex on \(\mathbb U\). It remains to prove coercivity. Since \(M_D\) is positive definite,
 its smallest eigenvalue \(\lambda_*\) is positive. Then
\[
\mathcal Q(u)
\ge
\frac{\lambda_*}{2}\|u_D\|^2
-\|c_1\|\,\|u_D\|+c_2,
\]
where $\|\cdot\|$ denotes the Euclidean norm of vectors.
Moreover, by the convexity of \(\mathcal{F}_p\), we have
$$\mathcal{F}_p(t)\geq \mathcal{F}_p(0)+f_p(0)t=f_p(0)t
\geq -|f_p(0)||t|.$$
Thus, for some constant \(C_0>0\),
$
\sum_{p\in D}\mathcal{F}_p(u_p)\geq -C_0\|u_D\|$.
Consequently,
\[
\mathcal{E}_\gamma(u) = \mathcal{Q}(u) + \sum_{p\in D} \mathcal{F}_p(u_p) \geq \frac{\lambda_*}{2}\|u_D\|^2 - (\|c_1\|+C_0)\|u_D\|+c_2.
\]
Hence \(\mathcal{E}_\gamma(u)\to+\infty\) as \(\|u_D\|\to\infty\), i.e.,
\(\mathcal{E}_\gamma\) is coercive on \(\mathbb{U}\).
\end{proof}

The next result gives the well-posedness of the direct problem.
\begin{theorem}\label{thm:well-posed}
Under Assumption \ref{assumption:monotone-f}, for a connected graph \(G\) with a nonempty boundary \(\partial D\),
problem~\eqref{eqn:Model_Nonlinear} admits a unique solution \(u\in\mathbb R^{\overline D}\).
\end{theorem}
\begin{proof}
By Lemma~\ref{lem:convexity-coercivity}, the energy \(\mathcal{E}_\gamma\) is continuous, coercive, and strictly convex on \(\mathbb{U}\).
Since \(\mathbb{U}\) is a closed affine subspace of the finite-dimensional space $\mathbb{R}^{\overline{D}}$, \(\mathcal{E}_\gamma\) attains its minimum
at some \(u^\ast\in \mathbb{U}\). For every \(v\in\mathbb R^{\overline D}\) with \(v|_{\partial D}=0\), the function $t\mapsto \mathcal{E}_\gamma(u^\ast+t v)$ has a minimum at \(t=0\). Hence
\[
\frac{\rm d}{{\rm d}t}\mathcal{E}_\gamma(u^\ast+t v)\Big|_{t=0}=0,
\]
which is equivalent to $(\nabla \mathcal{E}_\gamma(u^\ast))|_{D}=0$. Thus, by Lemma~\ref{Lem:full-energy-stationary}, \(u^\ast\) solves problem~\eqref{eqn:Model_Nonlinear}. The uniqueness follows from the strict convexity of \(\mathcal{E}_\gamma\) on \(\mathbb{U}\) in Lemma \ref{lem:convexity-coercivity}.
\end{proof}

\section{The discrete Calder\'{o}n problem and main results}\label{sec:inverse}

In this section, we describe the formulation of the inverse problem on square lattices, and state the main uniqueness results as well as the general analysis strategy.

\subsection{The discrete Calder\'{o}n problem on square lattices}
Now we study the discrete Calder\'{o}n problem associated with the model \eqref{eqn:Model_Nonlinear} on square lattices; see Fig. \ref{fig:lattice_corner}(a) for a schematic illustration. This setting is widely studied for the discrete Calder\'{o}n problem \cite{CurtisMorrow:1990,CurtisMorrow:1991}. The square lattice $G=(E,D,\partial D)$ is defined below. The sets $D$ and $\partial D$ of interior nodes and boundary nodes are given respectively by
 \[
 D = \{(i,j) \mid i,j = 1,\dots,n\}\ \mbox{and}\
 \partial D = \{(0,j), (n+1,j), (i,0), (i,n+1) \mid i,j = 1,\dots,n\}.
 \] The edge set \(E\) is given by
\begin{equation}
  E = \bigl\{\{p,q\} \mid p=(i,j)\in D,\; q=(i',j') \in \overline{D},\; |i-i'|+|j-j'|=1\bigr\}.\label{def:E}
\end{equation}
To employ the layer-stripping argument \cite{CurtisMorrow:1990,DengJin:2025} in the uniqueness proof, we define a diagonal layering of the lattice. For a fixed \(\ell \ge 0\) and for any node set \(X \in \{D, \partial D, \overline{D}\}\), we define the diagonal layer $X_\ell$ and its cumulative layer $X_{\le \ell}$ respectively by
\begin{equation}\label{eqn:X-ell}
X_\ell := \{(i,j) \in X \mid i+j = \ell\}\quad\mbox{and} \quad
X_{\leq \ell} := \{(i,j) \in X \mid i+j \leq \ell\},\quad
X_{>\ell}:=X\setminus X_{\leq \ell}.
\end{equation}
This diagonal decomposition plays a central role in the reconstruction of conductances layer by layer. Note that \(\partial D_{\le k}\) is the part of the boundary $\partial D$ adjacent to the chosen corner, and \(\partial D_{>k}\) is the remaining part of the boundary $\partial D$.

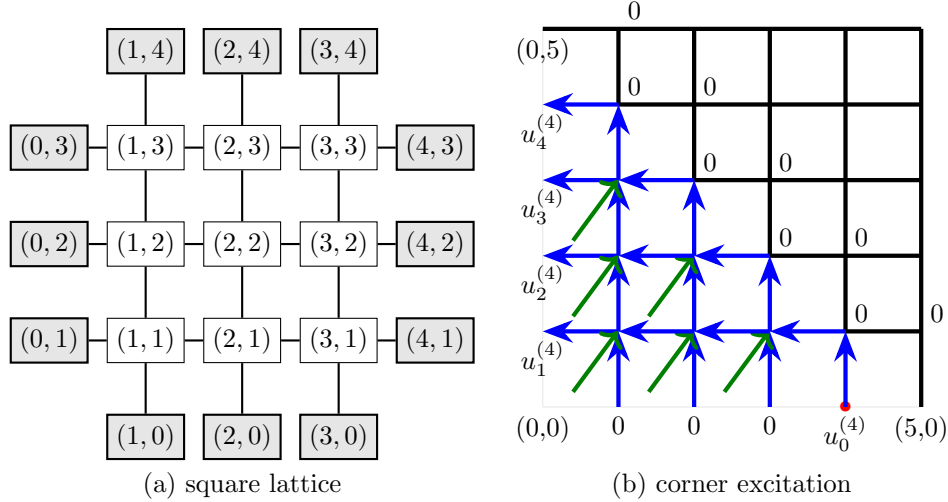
\begin{figure}[hbt!]
    \centering
\begin{tabular}{cc}
    \begin{tikzpicture}[
        node/.style={rectangle, draw, minimum size=2mm},
        boundary/.style={node, fill=gray!20, thick},
        interior/.style={node, fill=white},
        edge/.style={thick},
        scale = .85
    ]
    
    \def\n{3} 
    \def\spacing{1.5cm}
    
    \pgfmathtruncatemacro{\nmone}{\n-1}
     \pgfmathtruncatemacro{\nmine}{\n+1}
    \foreach \i in {1,...,\n} {
        \foreach \j in {1,...,\n} {
            \node[interior] (d\i\j) at ({\i*\spacing}, {\j*\spacing}) {$(\i,\j)$};
        }
    }
    
    \foreach \j in {1,...,\n} {
        \node[boundary] (t\j) at ({0},{\j*\spacing}) {$(0,\j)$};
    }
    
    \foreach \j in {1,...,\n} {
        \node[boundary] (b\j) at ({(\nmine)*\spacing},{\j*\spacing}) {$(\nmine,\j)$};
    }
    
    \foreach \i in {1,...,\n} {
        \node[boundary] (l\i) at ({\i*\spacing},{0}) {$(\i,0)$};
    }
    
    \foreach \i in {1,...,\n} {
        \node[boundary] (r\i) at ({\i*\spacing},{(\nmine)*\spacing}) {$(\i,\nmine)$};
    }
    
    
    \foreach \i in {1,...,\n} {
        \foreach \j in {1,...,\nmone} {
            \pgfmathtruncatemacro{\nextj}{\j+1}
            \draw[edge] (d\i\j) -- (d\i\nextj);
        }
    }
    
    \foreach \i in {1,...,\nmone} {
        \foreach \j in {1,...,\n} {
            \pgfmathtruncatemacro{\nexti}{\i+1}
            \draw[edge] (d\i\j) -- (d\nexti\j);
        }
    }
    
    \foreach \j in {1,...,\n} {
        \draw[edge] (t\j) -- (d1\j);
    }
    
    \foreach \j in {1,...,\n} {
        \draw[edge] (b\j) -- (d\n\j);
    }
    
    \foreach \i in {1,...,\n} {
        \draw[edge] (l\i) -- (d\i1);
    }
    
    \foreach \i in {1,...,\n} {
        \draw[edge] (r\i) -- (d\i\n);
    }
    
    

    \end{tikzpicture}
    & 
    \begin{tikzpicture}[scale=1.]
    \draw[gray!20, very thin] (0,0) grid (5,5);
    \draw[ultra thick](1,4)--(1,5);
        \draw[ultra thick](2,3)--(2,5);
        \draw[ultra thick](3,2)--(3,5);
        \draw[ultra thick](4,1)--(4,5);
        \draw[ultra thick](5,0)--(5,5);
\draw[ultra thick](4,1)--(5,1);
        \draw[ultra thick](3,2)--(5,2);
        \draw[ultra thick](2,3)--(5,3);
        \draw[ultra thick](1,4)--(5,4);
        \draw[ultra thick](0,5)--(5,5);
    \node[red] at (4,0) {$\bullet$} ;




    \node[below] at (0,1) {$u^{(4)}_1$};
    \node[below] at (0,2) {$u^{(4)}_2$};
    \node[below] at (0,3) {$u^{(4)}_3$};
    \node[below] at (0,4) {$u^{(4)}_4$};
    \node[above right] at (1,4) {0};
    \node[above right] at (2,3) {0};
    \node[above right] at (3,2) {0};
    \node[above right] at (4,1) {0};
        \node[above right] at (1,5) {0};
    \node[above right] at (2,4) {0};
    \node[above right] at (3,3) {0};
    \node[above right] at (4,2) {0};
     \node[above right] at (5,1) {0};
    \node[below] at (4,0) {$u^{(4)}_0$};
    \node[below] at (1,0) {0};
    \node[below ] at (2,0) {0};
    \node[below ] at (3,0) {0};
    \node[below] at (0,0) {(0,0)};
     \node[below] at (5,0) {(5,0)};
      \node[below] at (0,5) {(0,5)};
    \draw[ultra thick, blue, -{Stealth[length=4mm, width=3mm]}] (1,0) -- (1,1);
    \draw[ultra thick, blue, -{Stealth[length=4mm, width=3mm]}] (2,0) -- (2,1);
    \draw[ultra thick, blue, -{Stealth[length=4mm, width=3mm]}] (3,0) -- (3,1);
    \draw[ultra thick, blue, -{Stealth[length=4mm, width=3mm]}] (4,0) -- (4,1);

    \draw[ultra thick, blue, -{Stealth[length=4mm, width=3mm]}] (1,1) -- (0,1);
    \draw[ultra thick, blue, -{Stealth[length=4mm, width=3mm]}] (1,2) -- (0,2);
    \draw[ultra thick, blue, -{Stealth[length=4mm, width=3mm]}] (1,3) -- (0,3);
    \draw[ultra thick, blue, -{Stealth[length=4mm, width=3mm]}] (1,4) -- (0,4);

    \draw[ultra thick, blue, -{Stealth[length=4mm, width=3mm]}] (4-1+1,1) -- (4-1,1); 
    \draw[ultra thick, blue, -{Stealth[length=4mm, width=3mm]}] (3,1) -- (3,2);
    \draw[ultra thick, green!50!black, ->] (3-0.6,1-0.8) -- (3-0.01,1); 

    \draw[ultra thick, blue, -{Stealth[length=4mm, width=3mm]}] (4-2+1,1) -- (4-2,1); 
    \draw[ultra thick, blue, -{Stealth[length=4mm, width=3mm]}] (2,1) -- (2,2);
    \draw[ultra thick, green!50!black, ->] (2-0.6,1-0.8) -- (2-0.01,1); 

    \draw[ultra thick, blue, -{Stealth[length=4mm, width=3mm]}] (4-2+1,2) -- (4-2,2); 
    \draw[ultra thick, blue, -{Stealth[length=4mm, width=3mm]}] (2,2) -- (2,3);
    \draw[ultra thick, green!50!black, ->] (2-0.6,2-0.8) -- (2-0.01,2); 

    \draw[ultra thick, blue, -{Stealth[length=4mm, width=3mm]}] (4-3+1,1) -- (4-3,1); 
    \draw[ultra thick, blue, -{Stealth[length=4mm, width=3mm]}] (1,1) -- (1,2);
    \draw[ultra thick, green!50!black, ->] (1-0.6,1-0.8) -- (1-0.01,1); 

    \draw[ultra thick, blue, -{Stealth[length=4mm, width=3mm]}] (4-3+1,2) -- (4-3,2); 
    \draw[ultra thick, blue, -{Stealth[length=4mm, width=3mm]}] (1,2) -- (1,3);
    \draw[ultra thick, green!50!black, ->] (1-0.6,2-0.8) -- (1-0.01,2); 

    \draw[ultra thick, blue, -{Stealth[length=4mm, width=3mm]}] (4-3+1,3) -- (4-3,3); 
    \draw[ultra thick, blue, -{Stealth[length=4mm, width=3mm]}] (1,3) -- (1,4);
    \draw[ultra thick, green!50!black, ->] (1-0.6,3-0.8) -- (1-0.01,3); 
\end{tikzpicture}\\
(a) square lattice & (b) corner excitation
\end{tabular}
\caption{(a)
schematic illustration of  the square lattice: 
$\overline{D} = D \cup \partial D$, with interior nodes $D$ (white) and boundary nodes $\partial D$ (gray), and the edges connect nodes; and (b) corner excitations.}
\label{fig:lattice_corner}
\end{figure}

The goal of the discrete Calder\'{o}n problem is to determine the conductivity $\gamma \in \mathbb{R}_+^{E}$ from a finite set of DtN data pairs $(\varphi^{(i)}, F_{\gamma}(\varphi^{(i)}))_{i=1}^N$ (which is strictly less informative than the full nonlinear DtN map $F_{\gamma}(\varphi)$). The inverse problem is mathematically very challenging due to the nonlinearity of the direct problem, and the nonlinearity of the DtN map $F_{\gamma}(\varphi)$ with respect to the conductivity $\gamma$.

\subsection{Road map of the uniqueness analysis}

The analysis strategy depends on the type of available data. The common geometric ingredient in the first two settings is a diagonal layer-stripping procedure, pioneered by Curtis and Morrow \cite{CurtisMorrow:1990,CurtisMorrow:1991}. Specifically, starting from a corner of the lattice, one
uses boundary excitations whose associated potentials vanish beyond a prescribed diagonal. The discrete equations then determine the currents
across the next layer and, provided that the relevant potential differences do not vanish, the conductivities connecting new layers can be determined. Iterating this argument propagates the reconstruction through the entire lattice. More specifically, we develop the relevant uniqueness analysis in three steps.

First, in Section \ref{sec:corner-data}, we construct nonlinear corner excitations adapted to the unknown conductivity $\gamma$. The corresponding potentials are localized near one corner and lead to a constructive layer-by-layer recovery of the conductivity $\gamma$; see Theorem~\ref{thm:gamma-recovery} for the precise statement. This result provides an extension of the technique for the discrete Calder\'{o}n problem for the graph Laplacian \cite{CurtisMorrow:1990,CurtisMorrow:1991,DengJin:2025} to the semilinear case, and is important conceptually and technically. However, the prescribed excitations depend on the unknown conductivity $\gamma$. This leads to the question, answered in Section 6, whether finitely many a priori prescribed nonlinear measurements suffice.

Second, in Section~\ref{sec:linearized-recovery}, we investigate the case of linearizing the
nonlinear DtN map $F_\gamma(\varphi)$ at an arbitrary background boundary potential $\varphi$, which yields the linearized DtN map $\Lambda_\gamma(\varphi)$. Although the linearized corner excitations still depend on $\gamma$ through $\Lambda_\gamma(\varphi)$, the linearity of the perturbation problem restores the linear subspace structure of the Curtis-Morrow construction. Consequently, the required corner excitations can be computed from the known linearized DtN map $\Lambda_\gamma(\varphi)$, without prior knowledge of $\gamma$. We prove that the Fr\'echet derivative
$D_\varphi F_\gamma(\varphi)$ coincides with the DtN map $\Lambda_{\gamma}(\varphi)$ of the linearized problem, and that $\Lambda_{\gamma}(\varphi)$ allows us to construct the required corner excitations directly. Together with one pair of nonlinear Cauchy data $(\varphi,F_\gamma(\varphi))$, the map $\Lambda_\gamma(\varphi)$ uniquely determines both the conductivity $\gamma$ and the background potential $u_{\gamma,\varphi}$; see Theorem~\ref{Thm:LinearizeApproach} for the precise statement. The proof relies on suitably adapting the layer-stripping argument, and the approach is constructive in nature. The argument also yields a local \(C^{r-1}\) left inverse for the ideal measurement map $H_0(\gamma,\varphi) = \bigl(\Lambda_\gamma(\varphi),F_\gamma(\varphi),\varphi\bigr)$,
and in particular shows that the differential \(D H_0\) has full rank on a suitable admissible set.

Third and last, in Section \ref{sec:finite-data}, we replace the ideal linearized datum \(\Lambda_\gamma(\varphi)\) by finite differences of the nonlinear DtN map \(F_\gamma(\varphi)\). This gives a finite-measurement result in the nonlinear case. Let \(m=|\partial D|\). Then for any normalized basis
\(\{v_1,\ldots,v_m\}\subset \mathbb{S}^{m-1}\) of \(\mathbb R^{\partial D}\), and the \(m+1\) nonlinear measurements
$F_\gamma(\varphi)$ and $F_\gamma(\varphi+t v_i)$ for $i=1,\ldots,m$, taken near any \(\varphi\in \mathbb{R}^{\partial D}\), uniquely determine the conductivity \(\gamma\) for all sufficiently small \(t>0\); see Theorem~\ref{thm:finite-nonlinear-data} for the precise statement. The proof combines the local left invertibility of $H_0$, Taylor expansion of the nonlinear DtN map \(F_\gamma(\varphi)\) in the boundary datum $\varphi$, and a quantitative stability result for injective maps with full-rank differential.

Collectively these three results provide affirmative answers to the uniqueness issues of the nonlinear discrete Calder\'{o}n problem using finitely many Cauchy data pairs $(\varphi^{(i)},F_\gamma (\varphi^{(i)}))_{i=1}^N$ under certain conditions on the boundary excitations $(\varphi^{(i)})_{i=1}^N$.

\section{Corner data and nonlinear layer stripping}
\label{sec:corner-data}

In this section we prove a uniqueness result for a specialized family of boundary data via a nonlinear layer stripping procedure. We use a class of ``lower-left corner data,'' for which the potential $u_{\gamma,\varphi}$ localizes near the lower-left corner of the lattice.
The localization property in Lemma \ref{lem:corner-data-support} shows that the potential $u$ actually vanishes in \(\overline D_{>k}\). Then the localization property and the equations on the two diagonals \(D_{k+1}\) and \(D_k\) yield a triangular recursion for the currents across the next layer, and allow recovering the conductivities on the current layer from Ohm's law. More specifically, the proof relies on
four auxiliary facts. First, corner data give the support localization in Lemma \ref{lem:corner-data-support}. Second, for each prescribed voltage on the lower side of
the corner, there is a unique complementary voltage on the left side that produces a compatible corner datum, cf. Lemma \ref{lem:Characterization lemma}. Third, once the conductivities have been recovered in a corner region, the potential in
that region can be reconstructed from the local Cauchy data, cf. Lemma \ref{lem:PotentialDetermination}. Finally, the potential does not vanish at any node of the active diagonal for the special corner data used in the reconstruction, cf. Lemma \ref{Lemma:NoneZero}. These ingredients are then combined in Theorem~\ref{thm:gamma-recovery}.

\begin{definition}[Lower-left corner data]\label{def:corner-data}
For \(1\leq k\leq n+1\), a boundary datum \(\varphi\in \mathbb R^{\partial D}\) is called a
lower-left \(k\)-corner datum for the conductivity \(\gamma\) if
$\varphi|_{\partial D_{>k}}=0$  and $F_\gamma(\varphi)|_{\partial D_{>k}}=0$.
The set of lower-left \(k\)-corner data is denoted by
\[
V_k\equiv V_{k,\gamma}:=\{\varphi\in \mathbb R^{\partial D}:
\varphi|_{\partial D_{>k}}=0,\ F_\gamma(\varphi)|_{\partial D_{>k}}=0\}.
\]
\end{definition}

\begin{assumption}\label{assumption:fp-zero}
For every \(p \in D\), \(f_p\) is continuous, non-decreasing, and satisfies \(f_p(0)=0\).
\end{assumption}

The definition of a corner datum $\varphi$ requires both the imposed voltage and the resulting boundary current vanish away from the corner. The next lemma shows that these two boundary conditions force the potential $u$ to vanish identically beyond the diagonal $D_k$; see Fig. \ref{fig:lattice_corner}(b) for a schematic illustration. This localization property lays the basis for the layer-stripping argument.

\begin{lemma}\label{lem:corner-data-support}
Let Assumption \ref{assumption:fp-zero} hold. For
\(1\leq k\leq n+1\), let \(\varphi\in V_{k}\), and \(u\) be the solution of problem \eqref{eqn:Model_Nonlinear} with the excitation  \(\varphi\). Then $u_p=0$ for every $p\in D_{>k}$. That is, $\operatorname{supp}(u)\subset D_{\leq k}\cup \partial D_{\leq k}$.
\end{lemma}
\begin{proof}
By the definition of $V_k$, \(\varphi\in V_k\) implies $u=0$ on $\partial D_{>k}$ and the Neumann data vanish there. We prove the claim by descending induction on the row index.
For every interior node $(i,j)\in D$ with $i+j>k$ and $j=n$, the top boundary node $(i,n+1)\in\partial D_{>k}$ has zero voltage and zero current. Since it has the singleton neighbor $(i,n)$, the boundary flux formula gives $u_{(i,n)}=0$.

Now fix $j<n$ and assume that all interior nodes $(i',j')$ with $j'>j$ and $i'+j'>k$ have zero potential. Let $(i,j)\in D$ satisfy $i+j>k$. The node $p=(i,j+1)$ is an interior zero node by the induction hypothesis. Its other three neighbors, $(i-1,j+1)$, $(i+1,j+1)$, and $(i,j+2)$, are also zero: whenever a coordinate is on the boundary, this follows from the zero boundary data, and otherwise it follows from the induction hypothesis. Evaluating the equation at $p$ and using $f_p(0)=0$ gives
\[
0=\gamma_{p,(i,j)}(u_{(i,j)}-u_p),
\]
so $u_{(i,j)}=0$ because the conductance is positive. Descending over $j=n,n-1,\ldots,1$ proves $u=0$ on $D_{>k}$.
\end{proof}

Next we partition the set \(\partial D_{\leq k}\) into two parts:
\begin{equation*}
J_{k,1}:=\partial D_{\le k}\cap \{(i,j):j=0\}\quad\mbox{and} \quad
J_{k,2}:=\partial D_{\le k}\cap \{(i,j):i=0\},
\end{equation*}
i.e., the horizontal part and the vertical part of the boundary $\partial D_{\leq k}$.

The next lemma shows that any vector in the set \( V_k \) can be represented as a graph over $\mathbb{R}^{ J_{k,1}}$:  for any $\varphi_1\in \mathbb{R}^{ J_{k,1}}$, there exists a unique $\varphi_2\in \mathbb{R}^{ J_{k,2}}$ such that $(\varphi_1,\varphi_2) \in V_k $,
with $(\varphi_1,\varphi_2)$ being the Dirichlet boundary condition that takes the values given by $\varphi_1$ on  $J_{k,1}$, $\varphi_2$ on  $J_{k,2}$, and $0$ elsewhere.

\begin{lemma}\label{lem:Characterization lemma}
Under Assumption \ref{assumption:fp-zero}, for each \(1\leq k\leq n\) and each \(\varphi_{k,1} \in \mathbb R^{J_{k,1}}\), there exists a unique
$\varphi_{k,2}\equiv \varphi_{k,2}(\varphi_{k,1}, \gamma) \in \mathbb R^{J_{k,2}}$
such that $(\varphi_{k,1}, \varphi_{k,2}) \in V_k$. Moreover, if each \(f_p\) is \(C^r\), then the map $(\varphi_{k,1}, \gamma)\mapsto\varphi_{k,2}(\varphi_{k,1}, \gamma)$ is also \(C^r\).
\end{lemma}

\begin{proof}
Fix \(\varphi|_{J_{k,1}} = \varphi_{k,1}\). The construction proceeds
inductively from the diagonal \(k+1\) downwards to \(1\). Set
$u_p = 0$ for $p \in \overline{D}_{>k}$  and
$u_p = (\varphi_{k,1})_p$ for $p \in J_{k,1}$. Then for each \(k' = k+1, k, \dots, 1\), we traverse all interior nodes
\(p = (i,j)\) on \(D_{k'}\) with an increasing column index \(j\).
Note that equation \eqref{eqn:Model_Nonlinear} at $p$ reads
\[
\sum_{q \in N(p)} {\gamma}_{pq} (u_q - u_p) = f_p(u_p),
\]
with $N(p)=\{(i-1,j),(i,j+1),(i+1,j),(i,j-1)\}$. The potentials at the last three nodes have already been determined. Hence, by setting
\[
s_p:= \sum_{\substack{q \in N(p) \setminus\{(i-1,j)\}}}{\gamma}_{pq}(u_q-u_p),
\]
the equation at \(p\) reads
\[
{\gamma}_{p(i-1,j)} \bigl(u_{(i-1,j)}-u_p\bigr)+s_p=f_p(u_p)\quad\mbox{i.e.},\quad  u_{(i-1,j)} = u_p+\frac{f_p(u_p)-s_p}{{\gamma}_{p(i-1,j)}},
\]
since \({\gamma}_{p(i-1,j)}>0\). 
After finitely many steps all potential values on \(J_{k,2}\) can be determined.
Set $\varphi_{k,2}:=u|_{J_{k,2}}$, and let \(\varphi=(\varphi_{k,1},\varphi_{k,2})\in\mathbb R^{\partial D}\).
By the construction, \(u\) satisfies \eqref{eqn:Model_Nonlinear} at all nodes in $D_{\leq k+1}$. For any $p\in
D_{>k+1}$, $u_q=0$ for all $q\in N(p)$, and the equation
also holds because \(f_p(0)=0\). Thus \(u\) is the unique solution of problem \eqref{eqn:Model_Nonlinear} with the boundary value \(\varphi\), cf. Theorem \ref{thm:well-posed}. Moreover, since \(u=0\) on \(\overline D_{>k}\), $
F_\gamma(\varphi)|_{\partial D_{>k}}=0, $ and thus $(\varphi_{k,1},\varphi_{k,2})\in V_k$.
If another complementary datum $\widetilde \varphi_{k,2}\in \mathbb{R}^{J_{k,2}}$ produced a datum in $V_k$, Lemma \ref{lem:corner-data-support} would force its solution to vanish on 
$D_{>k}$. Then the preceding recursion, beginning with the prescribed $\varphi_{k,1}$, determines successively all remaining values, including $\varphi_{k,2}$. Hence the complement is unique. Finally, each step involves only arithmetic operations, division by
positive conductivities, and evaluations of \(f_p\). Since there are only finitely many steps, the map $(\varphi_{k,1},\gamma) \mapsto \varphi_{k,2}(\varphi_{k,1},\gamma)$
is the composition of finitely many \(C^r\) operations when each \(f_p\) is \(C^r\). Hence the resulting map is \(C^r\).
\end{proof}

\begin{remark}
By Lemma \ref{lem:Characterization lemma}, for each \(1\leq k\leq n\), the set \(V_k\) is the graph of a unique map
$ \varphi_{k,1}\mapsto \varphi_{k,2}(\varphi_{k,1},\gamma)$. Equivalently, \(\varphi_{k,2}\) is the unique element of
\(\mathbb R^{J_{k,2}}\) such that, for the boundary datum \(\varphi\) defined by
$\varphi|_{J_{k,1}}=\varphi_{k,1}$, $\varphi|_{J_{k,2}}=\varphi_{k,2}$ and $\varphi|_{\partial D_{>k}}=0$, there holds
$F_\gamma(\varphi)|_{\partial D_{>k}}=0$.
\end{remark}

When each \(f_p\) is linear, i.e., $f_p(s)=c_p s$, the DtN map \(F_\gamma(\varphi)\) is linear in $\varphi$, denoted by
$\Lambda_\gamma:\mathbb R^{\partial D}\to \mathbb R^{\partial D}$. Then \(V_k\) is a linear subspace, and the map furnished in Lemma
\ref{lem:Characterization lemma}
 can be made more explicit. Let \(\mathfrak{R}_{>k}:\mathbb{R}^{\partial D}\to \mathbb {R}^{\partial D_{>k}}\) be the restriction map, and let \(\mathfrak{E}_{k,i}:\mathbb R^{J_{k,i}}\to \mathbb R^{\partial D}\), $i=1,2$, be the extension-by-zero maps. The next result gives an explicit representation for \(\varphi_{k,2}(\varphi_{k,1},\gamma)\). It plays a crucial role in Section \ref{sec:linearized-recovery}: after linearization,
it allows constructing the required corner excitations directly from $\Lambda_\gamma(\varphi)$.

\begin{lemma}\label{lem:phi2-linear-smooth}
Suppose that each \(f_p\) is linear, \(f_p(s)=c_p s\), with \(c_p\geq0\),
and let \(\Lambda_\gamma:\mathbb R^{\partial D}\to\mathbb R^{\partial D}\)
be the associated DtN map. For fixed \(1\leq k\leq n\), let $A=\mathfrak{R}_{>k}\Lambda_\gamma \mathfrak{E}_{k,1}$ and $
B=\mathfrak{R}_{>k}\Lambda_\gamma \mathfrak{E}_{k,2}$.
Then \(B\) is injective, \(\operatorname{range}(A)\subset
\operatorname{range}(B)\), and for every
\(\varphi_{k,1}\in\mathbb R^{J_{k,1}}\), the unique complementary datum
\(\varphi_{k,2}(\varphi_{k,1},\gamma)\in\mathbb R^{J_{k,2}}\) such that
$(\varphi_{k,1},\varphi_{k,2}(\varphi_{k,1},\gamma))\in V_k$
is given by
\begin{equation}\label{eqn:graph-corner-lin}
\varphi_{k,2}(\varphi_{k,1},\gamma) =-B^\dagger A\varphi_{k,1},
\quad \mbox{with }B^\dag=(B^\top B)^{-1}B^\top .
\end{equation}
Moreover, the set $\mathcal U_k := \{M\in\mathbb R^{\partial D\times\partial D}: B(M)=\mathfrak{R}_{>k}M\mathfrak{E}_{k,2}\ \text{is injective}\}$
is open, \(\Lambda_\gamma\in\mathcal U_k\), and with $A(M):=\mathfrak{R}_{>k}M \mathfrak{E}_{k,1}$, the map $(\varphi_{k,1},M)\mapsto -B(M)^\dag A(M) \varphi_{k,1}$ is \(C^\infty\) on \(\mathbb R^{J_{k,1}}\times\mathcal U_k\).
\end{lemma}
\begin{proof}
Any boundary datum $\varphi$ supported on \(\partial D_{\leq k}\) can be expressed as 
$\varphi=\mathfrak{E}_{k,1}\varphi_{k,1}+\mathfrak{E}_{k,2}\varphi_{k,2}$.
By the definition of \(V_k\), $\varphi$ belongs to \(V_k\) if and only if
$\mathfrak{R}_{>k}\Lambda_\gamma\varphi=0$:
\begin{align}\label{eqn:datum-decom}
A\varphi_{k,1}+B\varphi_{k,2}=0.
\end{align}
By Lemma~\ref{lem:Characterization lemma}, for any \(\varphi_{k,1}\in\mathbb R^{J_{k,1}}\), there exists a unique
\(\varphi_{k,2}\in\mathbb R^{J_{k,2}}\) such that \eqref{eqn:datum-decom} holds. By taking \(\varphi_{k,1}=0\), the uniqueness gives
$\ker(B)=\{0\}$, i.e., injectivity of \(B\). The existence statement, applied to an arbitrary
\(\varphi_{k,1}\), gives $\operatorname{range}(A)\subset \operatorname{range}(B).$ Since \(B\) is injective, \(B^\top B\) is invertible and the Moore-Penrose formula $B^\dagger=(B^\top B)^{-1}B^\top$ holds. Moreover, since \(A\varphi_{k,1}\in\operatorname{range}(B)\), the unique
solution of \eqref{eqn:datum-decom} is $\varphi_{k,2}=-B^\dagger A\varphi_{k,1}$.
This shows \eqref{eqn:graph-corner-lin}. Next we show the smoothness. The set \(\mathcal U_k\) is open since \(B(M)\) is injective if and only if \(B(M)^\top B(M)\) is invertible, which is an open condition. The injectivity of \(B\) shows \(\Lambda_\gamma\in\mathcal U_k\). For any \(M\in\mathcal U_k\), the identity $B(M)^\dagger=\bigl(B(M)^\top B(M)\bigr)^{-1}B(M)^\top$ holds.
The maps \(M\mapsto A(M)\) and \(M\mapsto B(M)\) are linear, and matrix inversion is \(C^\infty\) on the set of invertible matrices. Thus
$(\varphi_{k,1},M)\mapsto -B(M)^\dagger A(M)\varphi_{k,1}$ is \(C^\infty\) on \(\mathbb R^{J_{k,1}}\times\mathcal U_k\).
\end{proof}

Once the conductivities in a corner region have been recovered, the next layer-stripping step requires the potential values in the region. These interior potential values are not part of the measurement and must be reconstructed from the local Cauchy data (i.e., boundary voltages and currents). The next result shows that this can be done recursively using only the already known conductivities and local Cauchy data. For any node set \(X\subset\overline D\), let \(E(X)\) denote the set of edges whose endpoints lie entirely in \(X\).
\begin{lemma}\label{lem:PotentialDetermination}
Let \(2\leq \ell\leq n+1\). Suppose that \(f_p\) is known for all
\(p\in D_{\leq \ell-1}\), and that \(\gamma\) is known on all edges in
\(E(\overline{D}_{\leq \ell})\). Let
\(u\in \mathbb R^{\overline D}\) be a solution of
\eqref{eqn:Model_Nonlinear} with the Dirichlet data \(\varphi\) and Neumann data $\psi=F_\gamma(\varphi)$.
Then the restriction \(u|_{D_{\le \ell}}\) is uniquely determined by the local data
$(\gamma|_{E(\overline D_{\le \ell})}, \varphi|_{\partial D_{\le \ell-1}},\psi|_{\partial D_{\le \ell-1}})$. More precisely, there exists a map
$\mathcal D_\ell:\mathbb R_+^{E(\overline D_{\le \ell})} \times \mathbb R^{\partial D_{\le \ell-1}} \times \mathbb R^{\partial D_{\le \ell-1}} \to \mathbb R^{D_{\le \ell}}$
such that, whenever \((\gamma,\varphi,\psi)\) arise from a solution of problem \eqref{eqn:Model_Nonlinear}, there holds
\[
u|_{D_{\le \ell}} = \mathcal D_\ell\bigl( \gamma|_{E(\overline D_{\le \ell})}, \varphi|_{\partial D_{\le \ell-1}}, \psi|_{\partial D_{\le \ell-1}}\bigr).
\]
Moreover, if each \(f_p\), \(p\in D_{\leq \ell-1}\), is of class \(C^r\),
then the map \(\mathcal D_\ell\) is of class \(C^r\).
\end{lemma}
\begin{proof}
We determine the potential $u$ in the corner \(D_{\leq \ell}\) layer by layer from the local Cauchy data $(\varphi|_{\partial D_{\leq \ell-1}},\psi|_{\partial D_{\leq \ell-1}})$ on \(\partial D_{\leq \ell-1}\) and from the
conductances $\gamma$ on \(E(\overline D_{\leq \ell})\). The procedure proceeds in two stages. First, the potential values at the nodes adjacent to the lower and left boundaries are obtained directly from the Neumann data. The Neumann data $\psi_{(i,0)}$ at the node $(i,0)$ on the lower boundary is given by
\[
\psi_{(i,0)} = \gamma_{(i,0)(i,1)} \bigl(u_{(i,0)}-u_{(i,1)}\bigr), \quad i=1,\ldots,\ell-1,
\]
from which we directly obtain
\begin{equation}\label{eqn:pot-lower}
u_{(i,1)} = \varphi_{(i,0)} - \frac{\psi_{(i,0)}}{\gamma_{(i,0)(i,1)}}, \quad i=1,\ldots,\ell-1.
\end{equation}
Similarly, on the left boundary, we have
\[
u_{(1,j)}=\varphi_{(0,j)}- \frac{\psi_{(0,j)}}{\gamma_{(0,j)(1,j)}}, \quad j=2,\ldots,\ell-1.
\]
Now we determine the remaining potential values in \(D_{\leq \ell}\). For each $i=1,\ldots,\ell-2$,
the potential value \(u_{(i+1,1)}\) has already been determined, cf. \eqref{eqn:pot-lower}.
We determine the remaining potential values columnwise. Then for $j=2,\ldots,\ell-1-i$ (which 
traverses the remaining interior nodes in $D_{\leq \ell}$), assume that all values appearing below, except possibly
\(u_{(i+1,j)}\), have already been determined. At the node \((i,j)\),
which belongs to \(D_{\leq \ell-1}\), equation \eqref{eqn:Model_Nonlinear} reads
\[
\sum_{q\in N(i,j)} \gamma_{(i,j)q} \bigl(u_q-u_{(i,j)}\bigr) = f_{(i,j)}(u_{(i,j)}).
\]
Set
\[
s_{(i,j)} := \sum_{\substack{q\in N(i,j)\\ q\neq (i+1,j)}}\gamma_{(i,j)q} \bigl(u_q-u_{(i,j)}\bigr).
\]
The quantity \(s_{(i,j)}\) is known from either the given boundary data or the previous step. 
Consequently,
\[
u_{(i+1,j)} = u_{(i,j)} + \frac{f_{(i,j)}(u_{(i,j)})-s_{(i,j)}}{\gamma_{(i,j)(i+1,j)}}.
\]
The edge \(\{(i,j),(i+1,j)\}\) lies in
\(E(\overline D_{\leq \ell})\). Hence this formula uniquely determines
\(u_{(i+1,j)}\). The procedure determines $u_p$ for all \(p\in D_{\leq \ell}\). 
For every compatible triple \((\gamma,\varphi,\psi)\), this gives a well-defined map
\[
\mathcal D_\ell\bigl(\gamma|_{E(\overline D_{\le \ell})},\varphi|_{\partial D_{\le \ell-1}},
\psi|_{\partial D_{\le \ell-1}}\bigr) =u|_{D_{\leq \ell}}.
\]
Finally, the construction consists of finitely many additions, multiplications, divisions by positive conductivities, and evaluations of
\(f_p\), \(p\in D_{\leq \ell-1}\). Thus, if
\(f_p\) are \(C^r\), then the map \(\mathcal D_\ell\) is \(C^r\).
\end{proof}

The reconstruction formulas for the new conductivities divide the currents across the next layer by potential differences on the active
diagonal. The next lemma shows that a suitable corner excitation produces a potential that is nonzero at every node of the corresponding diagonal.

\begin{lemma}\label{Lemma:NoneZero}
Let Assumption~\ref{assumption:fp-zero} hold. Let \(1\leq k\leq n\). If \(\varphi \in V_k\) satisfies
$\varphi_{(k,0)}\neq 0$, then the solution \(u=u_{\gamma,\varphi}\) of problem \eqref{eqn:Model_Nonlinear} is nonzero on $\overline D_k$.
\end{lemma}
\begin{proof}
By the definition of \(V_k\), for \(\varphi\in V_k\), $u_p=\varphi_p=0$  for every $ p\in \partial D_{>k}$, and by Lemma~\ref{lem:corner-data-support}, $u_p=0$ for every $p\in D_{>k}$. Hence $u_p=0$  for every $p\in \overline D_{>k}$.  We prove that
\( u_{(k-i,i)}\neq 0, \, i=0,\ldots,k.\)
By assumption, $u_{(k,0)}=\varphi_{(k,0)}\neq 0$. For \(i=0,\ldots,k-1\), let
$p_i:=(k-i,i+1)$. Then \(p_i\in D_{k+1}\subset D_{>k}\), we have \(u_{p_i}=0\).
The four neighbours of \(p_i\) are
\[
(k-i,i),\quad (k-i-1,i+1),\quad (k-i+1,i+1), \quad (k-i,i+2).
\]
The last two lie in \(\overline D_{>k}\), and have zero potentials. Using  equation \eqref{eqn:Model_Nonlinear} at $p_i$ 
and noting the condition \(u_{p_i}=0\) and the assumption \(f_{p_i}(0)=0\), we obtain
\[
\gamma_{p_i,(k-i,i)}u_{(k-i,i)} + \gamma_{p_i,(k-i-1,i+1)}u_{(k-i-1,i+1)}=0
\]
namely 
\[
u_{(k-i-1,i+1)} = - \frac{\gamma_{p_i,(k-i,i)}}{\gamma_{p_i,(k-i-1,i+1)}}u_{(k-i,i)}.
\]
By unrolling the recursion, we obtain
\[
u_{(k-i,i)} = (-1)^i\prod_{r=0}^{i-1} \frac{ \gamma_{(k-r,r+1),(k-r,r)}}{\gamma_{(k-r,r+1),(k-r-1,r+1)}}u_{(k,0)},\quad i=1,\ldots,k.
\]
Since \(u_{(k,0)}\neq 0\) and $\gamma\in \mathbb{R}_+^{E}$, each $u_{(k-i,i)}$ is nonzero and \(u\) is nonzero on \(\overline D_k\).
\end{proof}

We now prove a uniqueness result for corner excitations using the layer-by-layer reconstruction. For each diagonal \(k\), we
prescribe a unit voltage at the lower boundary node \((k,0)\), denoted by $\mathbf e_k$, and use
Lemma~\ref{lem:Characterization lemma} to determine the unique
complementary voltage $\varphi_{k,2}$ that produces a \(k\)-corner
datum. Lemma~\ref{lem:corner-data-support} localizes the resulting potential, Lemma~\ref{lem:PotentialDetermination} reconstructs its
potential in the region whose conductivities are already known, and
Lemma~\ref{Lemma:NoneZero} ensures that the divisions used to recover
the next conductivities are legitimate.
\begin{theorem}\label{thm:gamma-recovery}
Let Assumption~\ref{assumption:fp-zero} hold. Let $\varphi^{(k)}=\mathfrak{E}_{k,1}\mathbf e_k + \mathfrak{E}_{k,2}\varphi_{k,2}(\mathbf e_k,\gamma)$, $k=1,\ldots, n$. Then the conductivity \(\gamma\) on
$E(\overline D_{\le n+1})$ is uniquely determined by  $M_- :=\{(\varphi^{(k)},\psi^{(k)}=F_\gamma(\varphi^{(k)}))\}_{k=1}^n$.
Moreover, if each \(f_p\) is of class \(C^r\), then for each compatible family \(M_-\), the recovery is implemented locally by a \(C^r\) map
$\mathcal C:\mathcal U\subset\left(\mathbb R^{\partial D}\times\mathbb R^{\partial D}\right)^n\to\mathbb R^{E(\overline D_{\le n+1})},$
defined on a neighborhood \(\mathcal U\) of \(M_-\), such that $\mathcal C(M_-)=\gamma|_{E(\overline D_{\le n+1})}.$
\end{theorem}

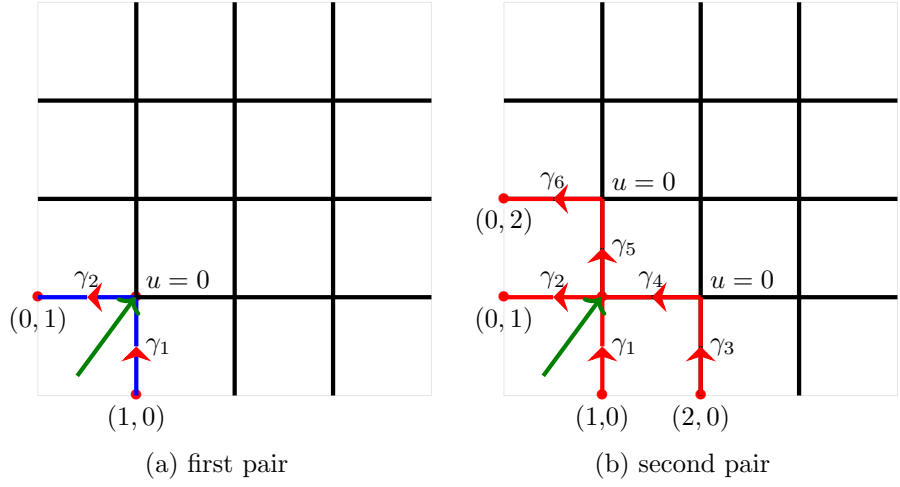
\begin{figure}[hbt!]
    \centering
\begin{tabular}{cc}
 \begin{tikzpicture}[scale=1.3]
    \draw[gray!20, very thin] (0,0) grid (4,4);
    \node[red] at (1,0) {$\bullet$} ;
     \node[red] at (0,1) {$\bullet$} ;
     \node[red] at (1,1) {$\bullet$} ;

    \node[below] at (0,1) {$(0,1)$};

    \node[above right] at (1,1) {$u=0$};

    \node[below] at (1,0) {$(1,0)$};

    \draw[ultra thick, blue, -{Stealth[length=2mm, width=4mm,red]}] (1,0) -- (1,0.5);\draw[ultra thick, blue] (1,0.5) -- (1,1);\node[right] at (1,0.5) {${\gamma}_1$};

    \draw[ultra thick, blue, -{Stealth[length=2mm, width=4mm,red]}] (1,1) -- (0.5,1); \draw[ultra thick, blue] (0.5,1) -- (0,1);\node[above] at (0.5,1) {${\gamma}_2$};
    \draw[ultra thick] (1,1)--(1,4);
      \draw[ultra thick] (2,0)--(2,4);
    \draw[ultra thick] (3,0)--(3,4);
 \draw[ultra thick] (1,1)--(4,1);
 \draw[ultra thick] (0,2)--(4,2);
 \draw[ultra thick] (0,3)--(4,3);

    \draw[ultra thick, green!50!black, ->] (1-0.6,1-0.8) -- (1-0.01,1); 

\end{tikzpicture}    & \begin{tikzpicture}[scale=1.3]

    \draw[gray!20, very thin] (0,0) grid (4,4);
    \node[red] at (1,0) {$\bullet$} ;
     \node[red] at (0,1) {$\bullet$} ;
     \node[red] at (1,1) {$\bullet$} ;
     \node[red] at (2,0) {$\bullet$} ;
     \node[red] at (0,2) {$\bullet$} ;
\draw[ultra thick] (1,1)--(1,4);
      \draw[ultra thick] (2,0)--(2,4);
    \draw[ultra thick] (3,0)--(3,4);
 \draw[ultra thick] (1,1)--(4,1);
 \draw[ultra thick] (0,2)--(4,2);
 \draw[ultra thick] (0,3)--(4,3);

    \node[below] at (0,1) {$(0,1)$};
 \node[below] at (0,2) {$(0,2)$};
    \node[above right] at (2,1) {$u=0$};
       \node[above right] at (1,2) {$u=0$};

    \node[below] at (1,0) {(1,0)};
    \node[below] at (2,0) {$(2,0)$};
    \draw[ultra thick, red, -{Stealth[length=2mm, width=4mm,red]}] (1,0) -- (1,0.5);\draw[ultra thick, red] (1,0.5) -- (1,1);\node[right] at (1,0.5) {${\gamma}_1$};
    \draw[ultra thick, red, -{Stealth[length=2mm, width=4mm,red]}] (2,0) -- (2,0.5);\draw[ultra thick, red] (2,0.5) -- (2,1);\node[right] at (2,0.5) {${\gamma}_3$};

    \draw[ultra thick, red, -{Stealth[length=2mm, width=4mm,red]}] (1,1) -- (0.5,1); \draw[ultra thick, red] (0.5,1) -- (0,1);\node[above] at (0.5,1) {${\gamma}_2$};
\draw[ultra thick, red, -{Stealth[length=2mm, width=4mm,red]}] (2,1) -- (1.5,1); \draw[ultra thick, red] (1.5,1) -- (1,1);\node[above] at (1.5,1) {${\gamma}_4$};
  \draw[ultra thick, red, -{Stealth[length=2mm, width=4mm,red]}] (1,2) -- (0.5,2); \draw[ultra thick, red] (0.5,2) -- (0,2);\node[above] at (0.5,2) {${\gamma}_6$};
  \draw[ultra thick, red, -{Stealth[length=2mm, width=4mm,red]}] (1,1) -- (1,1.5);\draw[ultra thick, red] (1,1.5) -- (1,2);\node[right] at (1,1.5) {${\gamma}_5$};

    \draw[ultra thick, green!50!black, ->] (1-0.6,1-0.8) -- (1-0.01,1); 

\end{tikzpicture} \\
  (a) first pair   & (b) second pair
\end{tabular}
  \caption{The first two pairs of DtN data in \(M_-\)}
    \label{fig:illustration_1}
\end{figure}
\begin{proof}
For each \(k=1,\ldots,n\), let $u^{(k)}:=u_{\gamma,\varphi^{(k)}}$ be the solution of problem \eqref{eqn:Model_Nonlinear} with the Dirichlet boundary value
\(\varphi^{(k)}\). Since $\varphi^{(k)}\in V_k$, Lemma~\ref{lem:corner-data-support} gives $u^{(k)}=0$ on $\overline D_{>k}$. We prove the recovery by induction on the diagonal layers.

\medskip
\noindent
\emph{Base step.} Consider the pair $(\varphi^{(1)},\psi^{(1)})\in M_-$. Since \(u^{(1)}=0\) on \(\overline D_{>1}\), we have $u^{(1)}_{(1,1)}=0$. 
The choice $u^{(1)}_{(1,0)} =\varphi^{(1)}_{(1,0)}=1$ and Lemma~\ref{Lemma:NoneZero} imply $u^{(1)}_{(1,0)}\neq0$ and $u^{(1)}_{(0,1)}\neq0$.
Now by definition, we have 
\begin{align*}
\psi^{(1)}_{(1,0)}&=\gamma_{(1,0)(1,1)}(u^{(1)}_{(1,0)}-u^{(1)}_{(1,1)})=\gamma_{(1,0)(1,1)}u^{(1)}_{(1,0)},\\
\psi^{(1)}_{(0,1)}&=\gamma_{(0,1)(1,1)}(u^{(1)}_{(0,1)}-u^{(1)}_{(1,1)})=\gamma_{(0,1)(1,1)}u^{(1)}_{(0,1)}.
\end{align*}
Consequently, we obtain
\[
\gamma_{(1,0)(1,1)}=\frac{\psi^{(1)}_{(1,0)}}{u^{(1)}_{(1,0)}}\quad\mbox{and}\quad \gamma_{(0,1)(1,1)}=\frac{\psi^{(1)}_{(0,1)}}{u^{(1)}_{(0,1)}}.
\]
Thus \(\gamma\) is recovered on $E(\overline D_{\le 2}).$ See Fig. \ref{fig:illustration_1}(a) for an illustration. This proves the base case.

\medskip\noindent\emph{Induction step.}
Assume that, for some \(2\leq \ell\leq n\), the conductivity has already been recovered on $
E(\overline D_{\le \ell}).$ Next we show how to recover \(\gamma\) on
$E(\overline D_{\le \ell+1})$ using the pair $(\varphi^{(\ell)},\psi^{(\ell)})\in M_-$.

\medskip
\noindent
\underline{Step 1: Recovery of the potential \(u^{(\ell)}\).} Since \(\gamma\) is known on \(E(\overline D_{\le \ell})\), Lemma
\ref{lem:PotentialDetermination} applied to the local Cauchy data $(\varphi^{(\ell)}|_{\partial D_{\le \ell-1}},\psi^{(\ell)}|_{\partial D_{\le \ell-1}})$ uniquely determines $u^{(\ell)}|_{D_{\le \ell}}$. Moreover, since \(\varphi^{(\ell)}\in V_\ell\), we also have $u^{(\ell)}=0$  on $\overline D_{>\ell}$.
Therefore \(u^{(\ell)}\) is known on all nodes needed in the next step.

\medskip
\noindent\underline{Step 2: Recovery of the currents across the next layer.} Given $\gamma$ and $u^{(\ell)}$ on $\overline{D}_{\leq \ell}$, we first recover the currents on each edge in $E(\overline{D}_\ell\cup\overline{D}_{\ell+1})$.
For an oriented edge \(pq\), define the current $\eta_{pq}:=\gamma_{pq}\bigl(u^{(\ell)}_p-u^{(\ell)}_q\bigr)$. Then the equation at \(p\) equivalently reads
\[
\sum_{q\in N(p)}\eta_{pq} =-f_p(u^{(\ell)}_p).
\]
Let \(q=(i,j)\in D_{\ell+1}\). Since $u^{(\ell)}_q=0$ and \(f_q(0)=0\), the equation at \(q\) implies that the two currents from
the diagonal $\overline D_\ell$ into \(q\) cancel out. Let $\xi_{(i,j)}:= \eta_{(i-1,j),(i,j)}$. Then
\[
\eta_{(i,j-1),(i,j)} = -\xi_{(i,j)}.
\]
At the two ends of the layer, these currents are determined by
the measured Neumann data: $\xi_{(1,\ell)}= \psi^{(\ell)}_{(0,\ell)} \text{ and }
\xi_{(\ell,1)} = -\psi^{(\ell)}_{(\ell,0)}.$
Now let \(p=(i,j)\in D_\ell\). The two currents from \(p\) to the layer $\overline{D}_{\ell-1}$ 
are already known, since \(\gamma\) and \(u^{(\ell)}\) are known on \(\overline D_{\le \ell}\).
The remaining two currents are those from \(p\) to \((i+1,j)\) and \((i,j+1)\). In terms of the
variables \(\xi\), the equation at \(p\) becomes
\begin{align}\label{eqn:eq-xi}
-f_p(u^{(\ell)}_p) =\xi_{(i+1,j)} - \xi_{(i,j+1)} + \eta_{p,(i-1,j)} + \eta_{p,(i,j-1)}.
\end{align}
The last two terms are known. Thus \eqref{eqn:eq-xi} determines one new
value of \(\xi\) from the preceding one. Starting from the known value
$\xi_{(\ell,1)}=-\psi^{(\ell)}_{(\ell,0)}$, we successively recover all currents
$\xi_{(i,j)} \text{ with }i+j=\ell+1.$ The final value agrees with $\xi_{(1,\ell)}=\psi^{(\ell)}_{(0,\ell)}$ since the data are compatible.

\medskip
\noindent
\underline{Step 3: Recovery of the conductivities in $E(\overline D_{\le \ell+1}) \setminus
E(\overline D_{\le \ell})$.} The edges $E(\overline D_{\le \ell+1}) \setminus
E(\overline D_{\le \ell})$ connect \(\overline D_\ell\) to \(\overline{D}_{\ell+1}\). Let $q=(i,j)\in D_{\ell+1}$.
Since $u^{(\ell)}_q=0$, we recover the conductance on the left edge by
\[
\gamma_{(i-1,j),(i,j)}
=
\frac{\xi_{(i,j)}}{u^{(\ell)}_{(i-1,j)}}.
\]
Similarly, using $\eta_{(i,j-1),(i,j)}
=
-\xi_{(i,j)}$, we recover the conductance on the lower edge by
\[
\gamma_{(i,j-1),(i,j)}=-\frac{\xi_{(i,j)}}{u^{(\ell)}_{(i,j-1)}}.
\]
The denominators are nonzero by Lemma~\ref{Lemma:NoneZero}, since
$(i-1,j),\text{ and } (i,j-1)$ belong to \(\overline D_\ell\). Hence all conductivities on
$E(\overline D_{\le \ell+1})\setminus E(\overline D_{\le \ell})$ are recovered.
This completes the induction step. Therefore the conductivity \(\gamma\) is uniquely determined on
$E(\overline D_{\le n+1})$. Finally, assume that each \(f_p\) is \(C^r\). The preceding reconstruction
uses only finitely many additions, multiplications, divisions by nonzero
quantities, and evaluations of the \(C^r\) functions \(f_p\). Hence the
reconstruction is \(C^r\) on any open set of data for which the
denominators in the formulas remain nonzero. This gives the asserted local \(C^r\) recovery map.
\end{proof}

\begin{corollary}
Let $M_+$ be the family of upper-right corner data obtained by applying the preceding construction after the reflection $(i,j)\mapsto(n+1-i,n+1-j)$. Then its union with the lower-left corner data family $M_-$ uniquely determines the conductivity $\gamma$ on $E$.
\end{corollary}
\begin{proof}
The conductivity on the complementary upper-right region can be recovered
by the same argument, starting from the upper-right corner. This and Theorem \ref{thm:gamma-recovery} directly imply the assertion.
\end{proof}
\begin{remark}
The measurement family \(M_-\) is determined by the full nonlinear DtN map \(F_\gamma\) in the following sense. For each \(k=1,\ldots,n\), the datum \(\varphi^{(k)}\) is the unique element of \(V_k\) whose restriction to \(J_{k,1}\) is \(\mathbf e_k\).
Once \(\varphi^{(k)}\) is determined, the corresponding measurement is $\psi^{(k)}=F_\gamma(\varphi^{(k)})$.
\end{remark}

\section{Recovery from the linearized problem}
\label{sec:linearized-recovery}

The corner excitations used in Section \ref{sec:corner-data} are central to the layer-stripping argument, but they are strongly tied to the unknown conductivity $\gamma$. Now we investigate a formulation in terms of data that can be specified independently of \(\gamma\). The key idea is to linearize the map $F_\gamma(\varphi)$ at an arbitrary $\varphi\in\mathbb{R}^{\partial D}$. The derivative is the DtN map $\Lambda_\gamma(\varphi)$ of a linear discrete Schr\"{o}dinger-type equation whose reaction coefficient is $\zeta_p=f_p'(u_p)$. Note that \(\zeta_p\) is not known a priori, and coupled to the unknown background \(u\). 
In this section, we prove that the linearized DtN map $\Lambda_{\gamma}(\varphi)$ provides the linearized corner excitations and their boundary currents. Combined with one nonlinear Cauchy data pair, which is used to recover the background potential $u$ (and the reaction coefficient $\zeta_p$), these data determine both $u$ and $\gamma$.

The analysis requires the following assumption.

\begin{assumption}\label{Assumption: fp in C1}
For each \(p\in D\), \(f_p\in C^r\), \(r\geq 1\), and $f_p'(t)\geq 0$ for all \(t\in\mathbb R\).
\end{assumption}

For a fixed background  $u=u_{\gamma,\varphi} \in \mathbb{R}^{\overline{D}}$, the derivative $w=\left. \frac{\mathrm d}{\mathrm dt} u_{\gamma,\varphi +t h}\right|_{t=0} \in \mathbb{R}^{\overline{D}}$ associated with a boundary perturbation \(h\in\mathbb R^{\partial D}\) satisfies
\begin{equation}\label{eq:linearized}
\left\{\begin{aligned}
(L_\gamma w)_p&= f_p'(u_p)w_p, && p\in D,\\
 w_p&=h_p, && p\in\partial D.
\end{aligned}\right.
\end{equation}
The linearized DtN map at the boundary value \(\varphi\) is denoted by $ \Lambda_\gamma(\varphi) = D_\varphi F_\gamma(\varphi).$ It maps the boundary perturbation \(h\) to the Neumann data of the solution \(w\) of problem \eqref{eq:linearized}.

The next result proves the smooth dependence of $u_{\gamma,\varphi}$ and $F_\gamma(\varphi)$ on $\gamma$ and $\varphi$, which justifies the linearization procedure. It also identifies the derivative $D_\varphi F_\gamma(\varphi)$ of $F_\gamma(\varphi)$ with respect to $\varphi$ as the DtN map $\Lambda_{\gamma}(\varphi)$ of the linearized equation \eqref{eq:linearized}.
\begin{proposition}\label{prop:linearization}
Let \(f=(f_p)_{p\in D}\) satisfy Assumption~\ref{Assumption: fp in C1}. Fix $(\gamma_0,\varphi_0)\in \mathbb R_+^E\times \mathbb R^{\partial D}$. Then there exist neighborhoods $U_\gamma\subset \mathbb R_+^E$ and $ U_\varphi\subset \mathbb R^{\partial D}$  of \(\gamma_0\) and \(\varphi_0\), respectively, such that the following three statements hold.
\begin{enumerate}[label={\rm(\roman*)}]
\item The maps $(\gamma,\varphi)\mapsto u_{\gamma,\varphi}\in \mathbb{R}^{\overline{D}}$
and $(\gamma,\varphi)\mapsto F_\gamma(\varphi)\in\mathbb R^{\partial D}$ are of class \(C^r\) on \(U_\gamma\times U_\varphi\).

\item For every \((\gamma,\varphi)\in U_\gamma\times U_\varphi\) and  \(h\in\mathbb R^{\partial D}\), the derivative
$ w=\frac{\mathrm d}{\mathrm dt}u_{\gamma,\varphi+t h}|_{t=0}$ exists and satisfies the linearized problem \eqref{eq:linearized}.

\item For each fixed \(\gamma\in U_\gamma\), the map $\varphi\mapsto F_\gamma(\varphi)$ is Fr\'{e}chet differentiable on \(U_\varphi\), and
$D_\varphi F_\gamma(\varphi)=\Lambda_\gamma(\varphi)$,
where \(\Lambda_\gamma(\varphi)\) is the DtN map associated with problem \eqref{eq:linearized}. Moreover, the map $(\gamma,\varphi)\mapsto \Lambda_\gamma(\varphi)$ is of class \(C^{r-1}\) on \(U_\gamma\times U_\varphi\).
\end{enumerate}
\end{proposition}
\begin{proof}
To separate the interior components of a candidate solution $u\in \mathbb{R}^{\overline{D}}$ from the boundary data $\varphi$, we denote by $\overline{u}$ the restriction of $u$ to $D$, i.e., $\overline{u}=u|_D$, and likewise, for any pair $(\overline{u},\varphi)$, we define $u\in \mathbb{R}^{\overline{D}}$ by $u|_D=\overline{u}$ and $u|_{\partial D}:=\varphi$. We prove the three assertions separately.

\medskip

\noindent\textbf{(i)}
Fix $(\gamma_0,\varphi_0)\in \mathbb R_+^E\times\mathbb R^{\partial D}$ and define a map $G:
\mathbb R^D\times\mathbb R^{\partial D}\times\mathbb R_+^E \to \mathbb R^D$ by
\[
G(\overline{u},\varphi,\gamma)_p := \sum_{q\in N(p)} \gamma_{pq} (u_q-u_p) - f_p(u_p), \quad p\in D.
\]
Then the equation $G(\overline{u},\varphi,\gamma)=0$ is exactly the nonlinear model \eqref{eqn:Model_Nonlinear} (with the boundary value \(\varphi\)). By Theorem \ref{thm:well-posed}, there exists a unique solution $u_0\equiv u_{\gamma_0,\varphi_0}\in\mathbb{R}^{\overline{D}}$ satisfying
$G(\overline{u_0},\varphi_0,\gamma_0)=0$. We claim that \(D_{\overline{u}}G(\overline{u_0},\varphi_0,\gamma_0)\) is invertible. The derivative $D_{\overline{u}}G(\overline{u_0},\varphi_0,\gamma_0) $ in \(\overline{u}\) is given by
\[
D_{\overline{u}}G(\overline{u_0},\varphi_0,\gamma_0) = - M_D  - \mathrm{diag}(f_p'((u_0)_p)),
\]
where $M_D$ is the matrix representation of the graph Laplacian given in \eqref{eqn:graph-Lap}.
By the argument in Lemma \ref{lem:convexity-coercivity}, since $M_D$ is positive definite it follows that $-M_D$ is negative definite. Since $\operatorname{diag}(f_p'((u_0)_p))$ is positive semidefinite by Assumption \ref{Assumption: fp in C1}, the derivative $D_{\overline{u}}G(\overline{u_0},\varphi_0,\gamma_0)$ is negative definite and invertible. Since each \(f_p\) is \(C^r\), the map \(G\) is jointly \(C^r\) in $(\overline{u},\varphi,\gamma)$
near \((\overline{u_0},\varphi_0,\gamma_0)\). The implicit function theorem implies that there exist
neighborhoods $W\ni \overline{u_0}$, $U_\varphi\ni\varphi_0$ and $U_\gamma\ni\gamma_0$, and a unique \(C^r\) map $S:U_\gamma\times U_\varphi\to W$ such that
$S(\gamma_0,\varphi_0)=\overline{u_0}$ and
\[
G(S(\gamma,\varphi),\varphi,\gamma)=0,\quad \forall (\gamma,\varphi)\in U_\gamma\times U_\varphi.\]
By the uniqueness in Theorem \ref{thm:well-posed}, $S(\gamma,\varphi)=\overline{u_{\gamma,\varphi}}$. Thus the map
$(\gamma,\varphi)\mapsto \overline{u_{\gamma,\varphi}}$ is $C^r$, and the map $(\gamma,\varphi)\mapsto u_{\gamma,\varphi}$ is also \(C^r\). Meanwhile, with \(u=u_{\gamma,\varphi}\), the DtN map $F_\gamma(\varphi)$ is given by
\[
(F_\gamma(\varphi))_r= \sum_{q\in N(r)} \gamma_{rq} \bigl(u_r-u_q\bigr), \quad r\in\partial D.
\]
Since $u_r=\varphi_r$ for $r\in\partial D,$ the formula shows that $F_\gamma(\varphi)$ depends on \(\gamma\), \(\varphi\), and
\(u_{\gamma,\varphi}\) through \(C^r\) operations. Therefore $(\gamma,\varphi)\mapsto F_\gamma(\varphi)$
is \(C^r\) on \(U_\gamma\times U_\varphi\).

\medskip

\noindent\textbf{(ii)}
Let $(\gamma,\varphi)\in U_\gamma\times U_\varphi$ and set $u:=u_{\gamma,\varphi}$.
For any $h\in\mathbb R^{\partial D}$, the \(C^r\) regularity in part \({\rm (i)}\) implies that
$w = \left. \frac{\mathrm d}{\mathrm dt} u_{\gamma,\varphi+t h} \right|_{t=0}$ exists. Differentiating
$G(\overline{u_{\gamma,\varphi+t h}},\varphi+t h,\gamma)=0$ with respect to \(t\) at \(t=0\)  gives
\[
D_{\overline{u}}G(\overline{u},\varphi,\gamma)w + D_\varphi G(\overline{u},\varphi,\gamma)h =0.
\]
Upon rearrangement, we obtain
\[
\sum_{q\in N(p)} \gamma_{pq} \bigl(w_q-w_p\bigr) = f_p'(u_p)w_p, \quad \forall p\in D,
\]
with the boundary condition $w|_{\partial D}=h$. This is precisely problem \eqref{eq:linearized}, proving the assertion in part (ii).

\medskip
\noindent\textbf{(iii)}
By part \({\rm (i)}\), the map $(\gamma,\varphi)\mapsto F_\gamma(\varphi)$
is \(C^r\). Hence for each fixed \(\gamma\), the map $\varphi\mapsto F_\gamma(\varphi)$ is Fr\'{e}chet differentiable on \(U_\varphi\).
Let \(h\in\mathbb R^{\partial D}\), and let \(w\) be the solution of the linearized problem \eqref{eq:linearized}. Differentiating the DtN
formula gives, for each \(r\in\partial D\),
\begin{align*}
\bigl(D_\varphi F_\gamma(\varphi)h\bigr)_r &= \left. \frac{\mathrm d}{\mathrm dt}
(F_\gamma(\varphi+t h))_r \right|_{t=0} = \sum_{q\in N(r)}\gamma_{rq} \bigl(w_r - w_q\bigr).
\end{align*}
Since $w_r=h_r$  for $r\in\partial D$, the right-hand side is exactly the Neumann data of the
solution \(w\) of problem \eqref{eq:linearized}. Hence the identity $D_\varphi F_\gamma(\varphi)h = \Lambda_\gamma(\varphi)h$ holds, i.e., 
$D_\varphi F_\gamma(\varphi)=\Lambda_\gamma(\varphi)$.
Finally, because $(\gamma,\varphi)\mapsto F_\gamma(\varphi)$ is \(C^r\), its derivative with respect to
\(\varphi\) is \(C^{r-1}\) in $(\gamma,\varphi)$. Thus the map $(\gamma,\varphi)\mapsto \Lambda_\gamma(\varphi)$
is of class \(C^{r-1}\) on \(U_\gamma\times U_\varphi\).
\end{proof}

Suppose that \(f\) satisfies Assumption~\ref{Assumption: fp in C1}. If, in addition, $f_p(0)=0$ and $f_p'(0)=0$ for every $p\in D$, then for \(\varphi\equiv0\), the unique solution $u$ of \eqref{eqn:Model_Nonlinear} vanishes identically. Problem \eqref{eq:linearized} has a zero reaction term: $\zeta_p=f_p'(u_p)=f_p'(0)=0$, and the linearized DtN map \(\Lambda_\gamma(0)=D_\varphi F_\gamma(0)\) is the DtN map of the
standard graph Laplacian. Hence, by the uniqueness results \cite{CurtisMorrow:1990,CurtisMorrow:1991} for the standard discrete Calder\'{o}n problem, the conductivity \(\gamma\) is uniquely determined by \(\Lambda_\gamma(0)\). Hence if the nonlinear map \(F_\gamma(\varphi)\) is known in a neighborhood of \(\varphi\equiv0\), then \(\Lambda_\gamma(0)\) can be determined from \(F_\gamma\), and then \(\gamma\) can be recovered.

For a general $\varphi \in \mathbb{R}^{\partial D}$, problem \eqref{eq:linearized} involves an unknown reaction term $\zeta_p=f_p'(u_p)$, since  \(u|_D\) is unknown, which greatly complicates the analysis. However, the pair \((\gamma,\zeta)\) is not arbitrary, and the two components are coupled via the relation $\zeta_p=f_p'(u_p)$ and the known nonlinear Cauchy data $\varphi^{(0)},\psi^{(0)}=F_\gamma(\varphi^{(0)})$.
The next theorem shows that the triple $\bigl(\Lambda_\gamma(\varphi^{(0)}),\psi^{(0)},\varphi^{(0)}\bigr)$ determines both \(u\) and \(\gamma\).

\begin{theorem}\label{Thm:LinearizeApproach}
Let Assumption~\ref{Assumption: fp in C1} hold. Let $ u=u_{\gamma,\varphi^{(0)}}$
be the solution of problem \eqref{eqn:Model_Nonlinear} with the Dirichlet boundary data $\varphi^{(0)}\in\mathbb R^{\partial D}$, and set $\psi^{(0)}:=F_\gamma(\varphi^{(0)})$. Let $\Lambda_\gamma(\varphi^{(0)})=D_\varphi F_\gamma(\varphi^{(0)})$ be the DtN map of problem
\eqref{eq:linearized}. Then the pair $(u,\gamma)$ is uniquely determined by the triple
 $\bigl(\Lambda_\gamma(\varphi^{(0)}),\psi^{(0)},\varphi^{(0)}\bigr)$. Moreover, the reconstruction is implemented locally on compatible data by
a \(C^{r-1}\) map $\mathcal R:\mathcal U\subset\mathbb R^{\partial D\times \partial D} \times
\mathbb R^{\partial D}\times \mathbb R^{\partial D}\to\mathbb R^{\overline D}\times \mathbb R_+^E$
such that
\[
\mathcal {R}\bigl( \Lambda_\gamma(\varphi^{(0)}),\psi^{(0)},\varphi^{(0)}\bigr)=(u,\gamma).
\]
\end{theorem}
\begin{proof}
We give a constructive proof for the lower-left half of the lattice. The proof consists of two steps.

\smallskip
\noindent
\underline{Step 1: Construction of linearized corner data from
\(\Lambda_\gamma(\varphi^{(0)})\).}
Let
$\Lambda:=\Lambda_\gamma(\varphi^{(0)})$. For each \(\ell=1,\ldots,n\), set $h_{\ell,1}:=\mathbf e_\ell\in\mathbb R^{J_{\ell,1}},
$ where \(\mathbf e_\ell\) is the coordinate vector localized at the node \((\ell,0)\in \partial D\). By Lemma
\ref{lem:phi2-linear-smooth} (on problem \eqref{eq:linearized} with $\zeta_p=f'_p(u_p)\geq 0$), we can construct $h_{\ell,2}\in\mathbb R^{J_{\ell,2}}$ using the linearized DtN matrix \(\Lambda\) such that $h^{(\ell)}:=\mathfrak{E}_{\ell,1} h_{\ell,1} + \mathfrak{E}_{\ell,2}h_{\ell,2}$ is a linearized lower-left \(\ell\)-corner datum. Let $z^{(\ell)}:=\Lambda h^{(\ell)}$. The maps $\Lambda\mapsto h^{(\ell)}$, and $\Lambda\mapsto z^{(\ell)}$ are both \(C^\infty\). Therefore, the construction $\Lambda \mapsto \{(h^{(\ell)},z^{(\ell)})\}_{\ell=1}^n$ is \(C^\infty\) near the compatible matrix \(\Lambda_\gamma(\varphi^{(0)})\).

\smallskip
\noindent
\underline{Step 2: Layerwise reconstruction of \(\gamma\) and \(u\).} We reconstruct the pair $(u,\gamma)$ by induction along the layer $\ell$. Let $w^{(\ell)}$ denote the solution of problem \eqref{eq:linearized} with \(h = h^{(\ell)}\). Since \(h^{(\ell)}\) is a linearized
corner datum, an analogue of Lemma~\ref{lem:corner-data-support} (for problem \eqref{eq:linearized}) gives $w^{(\ell)}=0$ on $\overline D_{>\ell}$, and moreover, Lemma~\ref{Lemma:NoneZero} and the condition \(h^{(\ell)}_{(\ell,0)}=1\) ensure $w^{(\ell)}\neq0$ on the diagonal $\overline D_\ell$.

\noindent \emph{Base step.} Consider the first corner pair $(h^{(1)},z^{(1)})$. Since $w^{(1)}_{(1,1)}=0$, by definition, we have
\begin{align*}
z^{(1)}_{(1,0)}&= \gamma_{(1,0)(1,1)} \bigl( h^{(1)}_{(1,0)}-w^{(1)}_{(1,1)}\bigr) = \gamma_{(1,0)(1,1)}h^{(1)}_{(1,0)},\\
z^{(1)}_{(0,1)} &=\gamma_{(0,1)(1,1)}\bigl(h^{(1)}_{(0,1)}-w^{(1)}_{(1,1)}\bigr) = \gamma_{(0,1)(1,1)}h^{(1)}_{(0,1)}.
\end{align*}
Since $h^{(1)}_{(1,0)}$ and $h^{(1)}_{(0,1)}$ are nonzero by the nonvanishing property of the linearized corner datum, we get
\[
\gamma_{(1,0)(1,1)} = \frac{z^{(1)}_{(1,0)}}{h^{(1)}_{(1,0)}}
\quad\mbox{and}\quad
\gamma_{(0,1)(1,1)} = \frac{z^{(1)}_{(0,1)}}{h^{(1)}_{(0,1)}}.
\]
Then the background nonlinear Cauchy data determine \(u_{(1,1)}\):
\[
\psi^{(0)}_{(0,1)} = \gamma_{(0,1)(1,1)} \bigl(\varphi^{(0)}_{(0,1)}-u_{(1,1)} \bigr)\quad \mbox{i.e.,}\quad u_{(1,1)} = \varphi^{(0)}_{(0,1)} - \frac{\psi^{(0)}_{(0,1)}}{\gamma_{(0,1)(1,1)}}.
\]
Thus we have recovered $\gamma|_{E(\overline D_{\le2})}$ and $u|_{D_{\le2}}$.

\medskip
\noindent
\emph{Induction step.} Suppose that, for some \(2\leq\ell\leq n\), we have already recovered $\gamma|_{E(\overline D_{\le\ell})}$ and $u|_{D_{\le\ell}}$. Now we recover $\gamma|_{E(\overline D_{\le\ell+1})}$ and $u|_{D_{\le\ell+1}}$. This is further broken down into three substeps (i)--(iii) below.

\medskip
\noindent \underline{(i) Recovery of \(w^{(\ell)}|_{D_{\le\ell}}\).} Note that the reaction coefficient $\zeta_p:=f_p'(u_p)$ is known for \(p\in D_{\le\ell}\) (and thus also for $p\in D_{\le\ell-1}$), since \(u\) is known over $D_{\le \ell}$ by the induction hypothesis. Lemma~\ref{lem:PotentialDetermination} (on problem \eqref{eq:linearized}), with known $\gamma$ on \(E(\overline D_{\le\ell})\) and \(\zeta_p=f_p'(u_p)\) on \(D_{\le\ell-1}\), and the local linearized Cauchy data $h^{(\ell)}|_{\partial D_{\le\ell-1}}$ and $z^{(\ell)}|_{\partial D_{\le\ell-1}}$ uniquely determines $w^{(\ell)}|_{\overline{D}_{\le\ell}}$. The dependence of $w^{(\ell)}|_{\overline{D}_{\leq \ell}}$ on the data is \(C^{r-1}\), since \(\zeta_p=f_p'(u_p)\) depends \(C^{r-1}\)-smoothly on \(u_p\).

\medskip
\noindent\underline{(ii) Recovery of \(\gamma|_{E(\overline D_{\le\ell+1})}\).}
For an oriented edge \(pq\), we define the current $\eta_{pq}$ by Ohm's law: $\eta_{pq}:=\gamma_{pq}\bigl(w^{(\ell)}_p-w^{(\ell)}_q\bigr)$. Then the governing equation for $w^{(\ell)}$ at the node \(p\) is equivalent to
\begin{align}\label{eqn:govern-eta}
\sum_{q\in N(p)}\eta_{pq}=-\zeta_p w^{(\ell)}_p.
\end{align}
Take any node $q_0=(i,j)\in D_{\ell+1}$. Since $w^{(\ell)}_{q_0}=0$, we have \(\zeta_{q_0} w^{(\ell)}_{q_0}=0\), and equation \eqref{eqn:govern-eta} at \(q_0\) gives the cancellation relation between the two currents entering \(q_0\) from $D_{\ell}$. Indeed, upon letting $\xi_{(i,j)}:=\eta_{(i-1,j),(i,j)}$, we obtain the sequence
\[
\eta_{(i,j-1),(i,j)} = -\xi_{(i,j)},\quad (i,j)\in D_{\ell+1}.
\]
Note that $\xi_{(1,\ell)}=z^{(\ell)}_{(0,\ell)}$ and $\xi_{(\ell,1)}=-z^{(\ell)}_{(\ell,0)}$.
Now let $p=(i,j)\in D_\ell$. The currents from \(p\) to \(\overline D_{\ell-1}\) are already known, since both \(\gamma\) and \(w^{(\ell)}\) are known on \(\overline D_{\le\ell}\). The remaining two currents are those from \(p\) to \((i+1,j)\) and \((i,j+1)\in D_{\ell+1}\). Thus equation \eqref{eqn:govern-eta} at the node \(p\) gives a triangular recursion for all $\xi_{(i,j)}$ with $i+j=\ell+1$:
\[
\xi_{(i+1,j)} - \xi_{(i,j+1)}  = -\zeta_p w^{(\ell)}_p - (\eta_{p,(i-1,j)} + \eta_{p,(i,j-1)}).
\]
Now we can recover the conductivity $\gamma$ on $E(\overline D_{\le\ell+1})\setminus E(\overline D_{\le\ell})$. Since $ w^{(\ell)}=0$  on $\overline D_{>\ell}$, for \(q_0=(i,j)\in D_{\ell+1}\) we have $w^{(\ell)}_{q_0}=0$. Since $w^{(\ell)}_{(i-1,j)}$ and $w^{(\ell)}_{(i,j-1)}$ are nonzero by the nonvanishing property of the linearized corner data, cf. Lemma \ref{Lemma:NoneZero}, we derive
\[
\gamma_{(i-1,j),(i,j)}= \frac{\xi_{(i,j)}}{w^{(\ell)}_{(i-1,j)}} \quad \mbox{and}\quad \gamma_{(i,j-1),(i,j)} = -\frac{\xi_{(i,j)}}{w^{(\ell)}_{(i,j-1)}}.
\]
Hence \(\gamma\) is recovered on $E(\overline D_{\le\ell+1})$. The dependence is \(C^{r-1}\), because it involves \(\zeta_p=f_p'(u_p)\) and the previously reconstructed \(u\).

\medskip
\noindent\underline{(iii) Recovery of \(u|_{D_{\le\ell+1}}\).} Now that $\gamma|_{E(\overline D_{\le\ell+1})}$ is known (from step (ii)), Lemma~\ref{lem:PotentialDetermination} (on problem \eqref{eqn:Model_Nonlinear}), the Cauchy data $(\varphi^{(0)}|_{\partial D_{\le\ell}}, \psi^{(0)}|_{\partial D_{\le\ell}})$  determine $u|_{D_{\le\ell+1}}$. This completes the induction step. Therefore the lower-left part of \((u,\gamma)\) is uniquely recovered.

\noindent Repeating the construction from the upper-right corner gives the remaining conductivities and potentials. Hence the full pair $(u,\gamma)$ is uniquely determined by $\bigl(\Lambda_\gamma(\varphi^{(0)}),\psi^{(0)},\varphi^{(0)}\bigr).$
Finally, all steps involve finitely many arithmetic operations, divisions by nonzero quantities, the Moore--Penrose pseudoinverse formula (on a fixed full-rank set),
and evaluations of \(f_p'\). Thus the reconstruction map $\mathcal{R}$ depends \(C^{r-1}\)-smoothly on the compatible data. This gives the asserted local \(C^{r-1}\) reconstruction map. 
\end{proof}

Theorem~\ref{Thm:LinearizeApproach} gives a constructive proof of the uniqueness for compatible data. For the passage from linearized measurements to finitely many nonlinear measurements in Section \ref{sec:finite-data}, we also need a nondegeneracy property of the differential of the ideal measurement map
\begin{equation}\label{eqn:ideal-data}
H_0(\gamma,\varphi) := \bigl(\Lambda_\gamma(\varphi),F_\gamma(\varphi),\varphi\bigr), \quad \mbox{with }\Lambda_\gamma(\varphi):=D_\varphi F_\gamma(\varphi).
\end{equation}
Indeed, the constructive formulas in the proof of Theorem~\ref{Thm:LinearizeApproach} remain well defined in a
neighborhood of every compatible datum, since the rank and nonvanishing conditions used in the recovery are open. The next result extracts from the proof a local smooth left inverse. Differentiating the left-inverse identity implies that the map $H_0$ has full-rank differential.
\begin{proposition}[Local reconstruction near compatible data]
\label{prop:local-reconstruction}
Suppose that Assumption~\ref{Assumption: fp in C1} holds with \(r\geq 2\).  Then, for every \((\gamma_0,\varphi_0)\in
\mathbb R_+^E\times\mathbb R^{\partial D}\), there exist open
neighborhoods
$U_0\subset \mathbb R_+^E\times\mathbb R^{\partial D}$ and $V_0\subset
\mathbb R^{\partial D\times\partial D}
\times\mathbb R^{\partial D}
\times\mathbb R^{\partial D}$,
with \((\gamma_0,\varphi_0)\in U_0\),
\(H_0(\gamma_0,\varphi_0)\in V_0\), and \(H_0(U_0)\subset V_0\), and a
map \(\mathcal R_0\in C^{r-1}(V_0;\mathbb R_+^E)\) such that
\[
\mathcal R_0\bigl(H_0(\gamma,\varphi)\bigr)=\gamma,
\quad
(\gamma,\varphi)\in U_0.
\]
Consequently, the map $\mathcal L_0:V_0\to\mathbb R_+^E\times\mathbb R^{\partial D}$, $\mathcal L_0(X,\psi,\varphi):=\bigl(\mathcal R_0(X,\psi,\varphi),\varphi\bigr)$, is a local \(C^{r-1}\) left inverse of \(H_0\), that is, $\mathcal L_0\bigl(H_0(\gamma,\varphi)\bigr) = (\gamma,\varphi)$ for $(\gamma,\varphi)\in U_0$, and moreover,
the differential $DH_0$ of $H_0$ is locally injective and has full column rank.
\end{proposition}
\begin{proof}
Fix \((\gamma_0,\varphi_0)\), and set \(X_0:=\Lambda_{\gamma_0}(\varphi_0)\) and \(\psi_0:=F_{\gamma_0}(\varphi_0)\). Thus \(H_0(\gamma_0,\varphi_0) =(X_0,\psi_0,\varphi_0)\). By the proof of Theorem~\ref{Thm:LinearizeApproach}, $\gamma$ is recovered from a triple \((X,\psi,\varphi)\) by finitely many explicit operations: additions, multiplications, evaluations of \(f_p\) and \(f_p'\), the Moore--Penrose
formula on matrices of fixed full column rank, and divisions by nonzero quantities (at the compatible datum \((X_0,\psi_0,\varphi_0)\)).
The full-rank condition and nonvanishing of the denominators are open conditions. Hence there exists an open neighborhood \(V_0\ni (X_0,\psi_0,\varphi_0)\) on which the reconstruction formulas are well-defined. Now we define \(\mathcal R_0(X,\psi,\varphi)\) on \(V_0\) to be the conductivity
$\gamma$ obtained from the process. Since the Moore--Penrose pseudoinverse is smooth
on the set of matrices with fixed full column rank (cf. Lemma \ref{lem:phi2-linear-smooth}), and since all other operations are \(C^{r-1}\), the map \(\mathcal R_0:V_0\to\mathbb R^E\) is \(C^{r-1}\). By shrinking \(V_0\), if necessary, we may also assume \(\mathcal R_0(V_0)\subset\mathbb R_+^E\), since \(\mathcal R_0(X_0,\psi_0,\varphi_0)=\gamma_0\in\mathbb R_+^E\). By Proposition~\ref{prop:linearization}, the map \(H_0(\gamma,\varphi)\)
is \(C^{r-1}\). Therefore, after possibly shrinking to an open neighborhood \(U_0\ni(\gamma_0,\varphi_0)\), we have
\(H_0(U_0)\subset V_0\). For every \((\gamma,\varphi)\in U_0\), the triple \(H_0(\gamma,\varphi)\) is compatible, and the reconstruction formula in
Theorem~\ref{Thm:LinearizeApproach} recovers $\gamma$. Thus
\[
\mathcal R_0\bigl(H_0(\gamma,\varphi)\bigr)=\gamma, \quad\forall (\gamma,\varphi)\in U_0.
\]
By definition,
\[
\mathcal L_0\bigl(H_0(\gamma,\varphi)\bigr) = \bigl(\mathcal R_0(H_0(\gamma,\varphi)),\varphi\bigr) = (\gamma,\varphi), \quad
\forall (\gamma,\varphi)\in U_0,
\]
i.e., \(\mathcal L_0\) is a local \(C^{r-1}\) left inverse of \(H_0\). 
Finally
\[
D\mathcal{L}_0\bigl(H_0(\gamma,\varphi)\bigr)\,
D H_0(\gamma,\varphi)
=
\operatorname{Id}_{\mathbb{R}^{E}\times\mathbb{R}^{\partial D}},
\]
from which we also have that $DH_0$ is injective and full-rank.
\end{proof}

\begin{remark}
The values on the diagonal can be obtained from the lower-left reconstruction and the upper-right reconstruction. Nonetheless, these values will coincide when the triple $(X,\psi,\varphi)$ is compatible.
\end{remark}

Proposition~\ref{prop:local-reconstruction} gives a local smooth left inverse $\mathcal{L}_0$ of the ideal measurement map $H_0$.   The perturbative argument in Section \ref{sec:finite-data} also requires injectivity on the entire admissible parameter set. The following corollary records the global injectivity, which follows directly from Theorem~\ref{Thm:LinearizeApproach}.

\begin{corollary}\label{Cor:Diff to gamma}
Under Assumption \ref{Assumption: fp in C1}, the map $H_0({\gamma},\varphi ) := \bigl(D_\varphi  {F}_{\gamma}(\varphi),  {F}_{\gamma}(\varphi), \varphi  \bigr)$ is $C^{r-1}$ and injective.
\end{corollary}
\begin{proof}
The \(C^{r-1}\)-regularity of the map \(H_0\) follows from Proposition~\ref{prop:linearization}. To prove injectivity, suppose that
$H_0(\gamma_1,\varphi_1) =H_0(\gamma_2,\varphi_2)$.
Component-wise comparison gives
\begin{align*}
\varphi_1=\varphi_2=:\varphi,\quad F_{\gamma_1}(\varphi) = F_{\gamma_2}(\varphi)=:\psi, \quad D_\varphi F_{\gamma_1}(\varphi) = D_\varphi F_{\gamma_2}(\varphi)=:\Lambda.
\end{align*}
Thus $\gamma_1$ and $\gamma_2$ generate the identical data triple $(\Lambda,\psi,\varphi)$. By Theorem~\ref{Thm:LinearizeApproach}, this triple uniquely determines the conductivity, and hence $\gamma_1=\gamma_2$, i.e., \(H_0\) is injective.
\end{proof}

\section{Recovery from finitely many nonlinear measurements}
\label{sec:finite-data}

In this section, we prove a uniqueness result for finitely many nonlinear measurements. The key idea is to replace the linearized DtN map $\Lambda_\gamma(\varphi)$ by finitely many evaluations of the nonlinear map $F_\gamma(\varphi)$. Specifically, let \(\{v_i\}_{i=1}^m\), \(m=|\partial D|\), be a set of boundary perturbation
directions that form an invertible matrix. Then intuitively, for sufficiently small \(t>0\), the difference quotients $\frac{F_\gamma(\varphi+t v_i)-F_\gamma(\varphi)}{t}$ approximate the vectors $D_\varphi F_\gamma(\varphi)v_i$, which form the linearized boundary data. The reference measurement map is $ H_0(\gamma,\varphi)=\bigl(D_\varphi F_\gamma(\varphi),F_\gamma(\varphi),\varphi\bigr)$. Corollary~\ref{Cor:Diff to gamma} shows that the map $H_0$ is injective, and Proposition~\ref{prop:local-reconstruction} gives a locally smooth left inverse of \(H_0\) and full rank of \(DH_0\). Thus the map \(H_0\) has the two properties needed for a perturbation argument: global injectivity and the nondegeneracy of the local differential.

It remains to show that these properties carry over when \(D_\varphi F_\gamma(\varphi)\) is replaced by finite differences. The next lemma gives a finite-dimensional quantitative stability result: an injective \(C^1\) map with full-rank differential on a compact set remains injective under sufficiently small \(C^1\) perturbations.

\begin{lemma}[Quantitative stability of injectivity]
\label{lem:stable-injectivity}
Let \(U\subset\mathbb R^d\) be open and convex, let \(K\Subset U\) be compact, and let \(G\in C^1(U;\mathbb R^N)\). Suppose that \(G|_K\) is injective and that \(DG(x)\) has rank \(d\) for every \(x\in K\). Then there exists an open convex neighborhood \(U_K\) of \(K\), with \(\overline U_K\subset U\), and a constant \(\delta>0\) such that every \(\Phi\in C^1(U_K;\mathbb R^N)\) satisfying $\|\Phi-G\|_{C^1(U_K)}<\delta$ is injective on \(K\).
\end{lemma}

\begin{proof}
Since \(DG(x)\) has rank \(d\) on \(K\), the smallest singular value \(\sigma_{\min}(DG(x))\) of $DG(x)$ is positive on \(K\). By the compactness of the set $K$, we have
\[
\mu:=\min_{x\in K}\sigma_{\min}(DG(x))>0.
\]
Since \(U\) is open and convex and \(K\Subset U\), we may choose an open convex neighborhood \(U_K\) of \(K\) such that \(\overline U_K\subset U\).
By the uniform continuity of \(DG\) on \(\overline U_K\), there exists
\(\rho>0\) such that
\[
x\in K,\ z\in U_K,\ |z-x|<\rho \quad\Longrightarrow\quad \|DG(z)-DG(x)\|<\frac{\mu}{4}.
\]
Take \(x,y\in K\) with \(|x-y|<\rho\). Since \(U_K\) is convex, the segment
\([x,y]\) lies in \(U_K\), and hence
\[
G(y)-G(x)=\int_0^1 DG(x+s(y-x))(y-x)\,{\rm d}s.
\]
Thus we obtain a local Lipschitz bound for $G$:
\[
\begin{aligned}
|G(y)-G(x)| &\ge |DG(x)(y-x)| -\int_0^1\|DG(x+s(y-x))-DG(x)\| |y-x|\,{\rm d}s  \\
&\ge \mu|y-x| - \frac{\mu}{4}|y-x| = \frac{3\mu}{4}|y-x|.
\end{aligned}
\]
Next set $A_\rho:=\{(x,y)\in K\times K:\ |x-y|\ge \rho\}$. If \(A_\rho\neq\emptyset\), then the injectivity of the map \(G|_K\) and the compactness of $A_\rho$ imply $\eta:=\min_{(x,y)\in A_\rho}|G(x)-G(y)|>0$.
If \(A_\rho=\emptyset\), then choose any \(\eta>0\). Set $\delta:=\min\left\{\frac{\mu}{4},\frac{\eta}{4}\right\}.$ Suppose that \(\Phi\in C^1(U_K;\mathbb R^N)\) satisfies \(\|\Phi-G\|_{C^1(U_K)}<\delta\). We show that \(\Phi|_K\) is injective. Take \(x,y\in K\), \(x\neq y\). If \(|x-y|<\rho\), then
\begin{align*}
|\Phi(y)-\Phi(x)| &\ge |DG(x)(y-x)| -\int_0^1\|D\Phi(x+s(y-x))-DG(x)\|\,|y-x|\,{\rm d}s \\
&\ge \mu|y-x|-\left(\|D\Phi-DG\|_{C^0(U_K)} +\frac{\mu}{4}\right)|y-x| \ge \frac{\mu}{2}|y-x|>0.
\end{align*}
Otherwise, if \(|x-y|\ge\rho\), then
\[
|\Phi(y)-\Phi(x)|\ge |G(y)-G(x)|-2\|\Phi-G\|_{C^0(U_K)}\ge \frac{\eta}{2}>0.
\]
Thus \(\Phi(x)\neq\Phi(y)\) whenever \(x\neq y\), and \(\Phi\) is injective on \(K\).
\end{proof}

The uniform version below will be applied to the compact family $\mathcal{P}_\varepsilon$ of normalized, uniformly nondegenerate direction matrices.
\begin{corollary}[Uniform stability for a compact family]
\label{cor:uniform-stable-injectivity} Let \(U\subset\mathbb R^d\) be open and convex, let \(K\Subset U\) be
compact, and let \(A\) be a compact metric space. Let \(G:U\times A\to\mathbb R^N\) be continuous, and assume that \(G(\cdot,\alpha)\in C^1(U;\mathbb R^N)\) for each \(\alpha\in A\), with \((x,\alpha)\mapsto D_xG(x,\alpha)\) continuous on \(U\times A\).
Suppose that \(G(\cdot,\alpha)|_K\) is injective for every \(\alpha\in A\), and that \(D_xG(x,\alpha)\) has rank \(d\) for every
\((x,\alpha)\in K\times A\). Then there exist an open convex neighborhood \(U_K\) of \(K\), with \(\overline U_K\subset U\), and a constant \(\delta>0\) such that, for every \(\alpha\in A\), every
\(\Phi\in C^1(U_K;\mathbb R^N)\) satisfying $\|\Phi-G(\cdot,\alpha)\|_{C^1(U_K)}<\delta$ is injective on \(K\).
\end{corollary}
\begin{proof}
The proof is identical to that of Lemma~\ref{lem:stable-injectivity}, with all constants chosen uniformly in \(\alpha\). Namely, set $\mu:=\min_{(x,\alpha)\in K\times A}\sigma_{\min}(D_xG(x,\alpha))>0$. Choose an open convex \(U_K\) with \(\overline U_K\subset U\). By the uniform continuity of \(D_xG\) on \(\overline U_K\times A\), choose \(\rho>0\) such that the estimate \(\|D_xG(z,\alpha)-D_xG(x,\alpha)\|<\frac{\mu}{4}\) holds whenever \(x\in K\), \(z\in U_K\), and \(|z-x|<\rho\). Define $A_\rho:=\{(x,y,\alpha)\in K\times K\times A:\ |x-y|\ge\rho\}$.
If \(A_\rho\neq\emptyset\), then the continuity, compactness, and injectivity of each \(G(\cdot,\alpha)|_K\) imply
\[
\eta:=\min_{(x,y,\alpha)\in A_\rho}|G(x,\alpha)-G(y,\alpha)|>0.
\]
If \(A_\rho=\emptyset\), choose any \(\eta>0\). Let $\delta=\min\left\{\frac{\mu}{4},\frac{\eta}{4}\right\}$. Then the argument in Lemma~\ref{lem:stable-injectivity} applies uniformly in \(\alpha\).
\end{proof}

We now apply the stability results in Lemma \ref{lem:stable-injectivity} and Corollary \ref{cor:uniform-stable-injectivity} to the nonlinear
DtN map $F_\gamma(\varphi)$.
The injectivity of the reference map $H_0$ follows from Theorem \ref{Thm:LinearizeApproach}, and its full-rank differential follows from
Proposition~\ref{prop:local-reconstruction}. Taylor expansion shows that the finite-difference map converges to $D_\varphi F_\gamma(\varphi)$ in \(C^1\), uniformly over compact parameter sets and uniformly over
all normalized direction matrices satisfying the prescribed uniform nondegeneracy condition. Finally, the stability lemma gives the following finite-data uniqueness result.

\begin{theorem}\label{thm:finite-nonlinear-data}
Suppose that Assumption~\ref{Assumption: fp in C1} holds with \(r\ge3\). Let \(K_\gamma\Subset\mathbb R_+^E\) and \(K_\varphi\Subset\mathbb R^{\partial D}\) be compact, and set \(K:=K_\gamma\times K_\varphi\). Let \(m:=|\partial D|\). For \(\varepsilon>0\), define
\[
\mathcal P_\varepsilon:=\left\{P=[v_1,\ldots,v_m]\in\mathbb R^{m\times m}:\|v_i\|=1,\ i=1,\ldots,m,\ |\det P|\ge\varepsilon\right\}.
\]
Then there exists \(t_0>0\), depending only on \(K_\gamma\), \(K_\varphi\) and \(\varepsilon\), such that for every \(P=[v_1,\ldots,v_m]\in\mathcal P_\varepsilon\) and every \(t\in(0,t_0]\), the following implication holds. If two pairs \((\gamma,\varphi),(\gamma',\varphi')\in K\) satisfy
\begin{align}\label{eqn:finite-data}
\varphi=\varphi',\quad F_\gamma(\varphi)=F_{\gamma'}(\varphi'),\quad F_\gamma(\varphi+t v_i)= F_{\gamma'}(\varphi'+t v_i), \quad i=1,\ldots,m,
\end{align}
then \(\gamma=\gamma'\).
\end{theorem}
\begin{proof}
Choose open bounded convex sets \(U_\gamma\Subset\mathbb R_+^E\) and \(U_\varphi\Subset\mathbb R^{\partial D}\) such that \(K_\gamma\subset U_\gamma\) and \(K_\varphi\subset U_\varphi\). Set \(U:=U_\gamma\times U_\varphi\). For \(P=[v_1,\ldots,v_m]\in\mathcal P_\varepsilon\) and \(t\ge0\), define
\[
\Theta_t^P(\gamma,\varphi):=\bigl(Q_t^P(\gamma,\varphi),F_\gamma(\varphi),\varphi\bigr),\quad (\gamma,\varphi)\in U,
\]
where, for \(t>0\), the \(i\)-th column of \(Q_t^P(\gamma,\varphi)\) is \(\frac{F_\gamma(\varphi+t v_i)-F_\gamma(\varphi)}{t},\) while \(Q_0^P(\gamma, \varphi):=D_\varphi F_\gamma(\varphi)P.\) Below we denote \(G_P:=\Theta_0^P\). The proof is divided into four steps.

\medskip
\noindent
\textbf{Step 1: uniform \(C^1\)-convergence of \(\Theta_t^P\) to \(G_P\).} By Proposition~\ref{prop:linearization}, the map \((\gamma,\varphi)\mapsto F_\gamma(\varphi)\) is \(C^r\) ($r\geq 3$) on compact subsets of \(\mathbb R_+^E\times\mathbb R^{\partial D}\). For each \(i\in\{1,\ldots,m\}\), Taylor's formula gives
\begin{align}\label{eqn:Taylor-Qt}
F_\gamma(\varphi+t v_i)=F_\gamma(\varphi)+tD_\varphi F_\gamma(\varphi)v_i +t^2\int_0^1(1-s) D_\varphi^2F_\gamma(\varphi+s t v_i)[v_i,v_i]\,{\rm d}s.
\end{align}
Consequently, $Q_t^P(\gamma,\varphi)$ can be represented by
\[
Q_t^P(\gamma,\varphi)=D_\varphi F_\gamma(\varphi)P+tR_t^P(\gamma,\varphi),
\]
where the \(i\)-th column of \(R_t^P\) is the integral remainder in \eqref{eqn:Taylor-Qt}. Fix \(\tau>0\). The set
\[
\widetilde U_\tau:=\{(\gamma,\varphi+s v): \gamma\in\overline U_\gamma, \varphi\in\overline U_\varphi,\  0\le s\le\tau,\  v\in \mathbb{S}^{m-1}\}
\]
is compact. Since the map \((\gamma,\varphi)\mapsto F_\gamma(\varphi)\) is \(C^3\), the map \(D_\varphi^2F\) is \(C^1\) on \(\widetilde U_\tau\). Therefore there exists \(C_\tau>0\), independent of \(P\in\mathcal P_\varepsilon\), such that
$\sup_{P\in\mathcal P_\varepsilon}\sup_{0<t\le\tau} \|R_t^P\|_{C^1(U)} \le C_\tau.$ Consequently,
\[
\sup_{P\in\mathcal P_\varepsilon} \|\Theta_t^P-G_P\|_{C^1(U)} \le C_\tau t, \quad 0<t\le\tau.
\]

\medskip
\noindent
\textbf{Step 2: the reference family \(\{G_P\}\) satisfies uniform stability.} Clearly, the set \(\mathcal P_\varepsilon\) is compact. Using the reference map $H_0(\gamma,\varphi) :=(\Lambda_\gamma(\varphi),F_\gamma(\varphi),\varphi)$, for each \(P\in\mathcal P_\varepsilon\),
\[
G_P(\gamma,\varphi) = (D_\varphi F_\gamma(\varphi)P,F_\gamma(\varphi),\varphi) = J_P(H_0(\gamma,\varphi)),
\]
with $J_P(X,\psi,\varphi):=(XP,\psi,\varphi)$. Since \(|\det P|\ge\varepsilon>0\), the map \(J_P\) is a linear isomorphism. By Corollary~\ref{Cor:Diff to gamma}, \(H_0\) is injective; hence \(G_P|_K\) is injective for every \(P\in\mathcal P_\varepsilon\). Moreover, Proposition~\ref{prop:local-reconstruction} gives a local \(C^{r-1}\) left inverse $\mathcal{L}_0$ of the map \(H_0\) near every point of \(K\). Since \(r\ge3\), the left inverse $\mathcal{L}_0$ is at least \(C^2\). Thus, differentiating the identity $\mathcal{L}_0\circ H_0=\operatorname{Id}$ shows that \(DH_0\) has full rank on \(K\). Since \(J_P\) is a linear isomorphism, \(DG_P\) also has full rank on \(K\) for every \(P\in\mathcal P_\varepsilon\).
Finally, the map \(((\gamma,\varphi),P)\mapsto DG_P(\gamma,\varphi)\) is continuous on \(U\times\mathcal P_\varepsilon\). Therefore the family \(\{G_P\}_{P\in\mathcal P_\varepsilon}\) satisfies the hypotheses of Corollary~\ref{cor:uniform-stable-injectivity}.

\medskip
\noindent
\textbf{Step 3: injectivity of \(\Theta_t^P\) for small \(t\).} By Corollary~\ref{cor:uniform-stable-injectivity}, there exist an open convex neighborhood \(U_K\) of \(K\), with \(\overline U_K\subset U\), and  a constant \(\delta_{\rm st}>0\) such that, for every \(P\in\mathcal P_\varepsilon\), any \(C^1\) map \(\Phi\) satisfying $\|\Phi-G_P\|_{C^1(U_K)}<\delta_{\rm st}$ is injective on \(K\). Choose \(t_0\in(0,\tau]\) so small that $C_\tau t_0<\delta_{\rm st}$.
Then, for every \(P\in\mathcal P_\varepsilon\) and every \(t\in(0,t_0]\), there holds
\[
\|\Theta_t^P-G_P\|_{C^1(U_K)}\le \|\Theta_t^P-G_P\|_{C^1(U)} \le C_\tau t<\delta_{\rm st}.
\]
Hence \(\Theta_t^P|_K\) is injective for every
\(P\in\mathcal P_\varepsilon\) and \(t\in(0,t_0]\).

\medskip
\noindent
\textbf{Step 4: conclusion from equality of finite data.} Suppose that \((\gamma,\varphi),(\gamma',\varphi')\in K\) satisfy \eqref{eqn:finite-data}. Then, for each \(i\),
\[
\frac{F_\gamma(\varphi+t v_i)-F_\gamma(\varphi)}{t} = \frac{F_{\gamma'}(\varphi'+t v_i)-F_{\gamma'}(\varphi')}{t}.
\]
Thus the finite-difference components of \(\Theta_t^P\) agree. This and the equalities of \(F_\gamma(\varphi)\) and \(\varphi\) give
$\Theta_t^P(\gamma,\varphi) = \Theta_t^P(\gamma',\varphi').$ Since \(\Theta_t^P\) is injective on \(K\) (from Step 3), we obtain \((\gamma,\varphi)=(\gamma',\varphi')\). In particular, \(\gamma=\gamma'\).
\end{proof}

\bibliographystyle{abbrv}
\bibliography{references}
\end{document}